\documentclass[11pt]{amsart}

\usepackage{amscd}
\usepackage{amsmath, amssymb}
\usepackage{amsfonts}

\renewcommand{\leq}{\leqslant}
\renewcommand{\geq}{\geqslant}

\newcommand{\be}{\begin{equation}}
\newcommand{\ee}{\end{equation}}

\begin{document}
\newtheorem{claim}{Claim}
\newtheorem{theorem}{Theorem}[section]
\newtheorem{lemma}[theorem]{Lemma}
\newtheorem{corollary}[theorem]{Corollary}
\newtheorem{proposition}[theorem]{Proposition}
\newtheorem{question}{question}[section]
\newtheorem{definition}[theorem]{Definition}
\newtheorem{remark}[theorem]{Remark}
\newtheorem{example}[theorem]{Example}
\newtheorem{problem}[theorem]{Problem}

\numberwithin{equation}{section}

\title[Dirichlet problem for fully non-linear elliptic equations]{Regularity of fully non-linear elliptic equations on Hermitian manifolds. III}

\author{Rirong Yuan} \address{School of Mathematics, South China University of Technology, Guangzhou  510641, China} \email{yuanrr@scut.edu.cn}
\thanks{The author is supported by the National Natural Science Foundation of China (Grant No. 11801587)}

\begin{abstract}
Under structural conditions which are almost 
optimal, we
 derive a quantitative version of boundary estimate then prove existence of solutions to Dirichlet problem for a class of
 fully nonlinear elliptic equations on Hermitian manifolds. 

 
\end{abstract}

 \maketitle


\section{Introduction}

Let $(M,J,\omega)$ be a compact Hermitian manifold of complex
dimension $n\geq 2$  with  boundary $\partial M$, $\bar M=M\cup\partial M$,
and $\omega= \sqrt{-1}  g_{i\bar j} dz^i\wedge d\bar z^j$ denote the K\"ahler form being compatible with the complex structure $J$.


Let $\Gamma\subset\mathbb{R}^n$ be an open symmetric convex cone containing positive cone
 \[\Gamma_n:=\{\lambda\in \mathbb{R}^n: \mbox{ each component } \lambda_i>0\}\subseteq\Gamma\]
  with vertex at the origin  and with the boundary $\partial \Gamma\neq \emptyset$.
  
  This article is a sequel to \cite{yuan2017,yuan2019}.
 The primary purpose of this paper is to study the following Dirichlet problem for standard equations
 \begin{equation}
\label{mainequ}
\begin{aligned}
 \,& f(\lambda(\mathfrak{g}[u]))= \psi   \mbox{ in } M, \,& u=  \varphi  \mbox{ on }\partial M,
\end{aligned}
\end{equation}
 which are determined by smooth symmetric
functions $f$, defined in $\Gamma$, of eigenvalues of complex Hessians,
where $\mathfrak{g}[u]=\chi+\sqrt{-1}\partial \overline{\partial} u,$
 ${\chi}$ is a smooth real $(1,1)$-form, $\psi$ and $\varphi$ are sufficiently smooth functions. 

 In a pioneer paper \cite{CNS3}, Caffarelli-Nirenberg-Spruck initiated the study of the Dirichlet problem
   of this type on bounded domains of real Euclidean spaces.
 Since then the equations of this type have been extensively studied in real and complex variables.
 In literature the hypotheses on $f$ 
 include
 \begin{equation}
 \label{elliptic}
 \begin{aligned}
\,& f_i(\lambda):=\frac{\partial f}{\partial \lambda_{i}}(\lambda)> 0  \mbox{ in } \Gamma,\,&  \forall 1\leq i\leq n,
\end{aligned}
\end{equation}
 \begin{equation}\label{concave} \begin{aligned}
 f \mbox{ is  concave in } \Gamma,
 \end{aligned} \end{equation}
   \begin{equation}
\label{addistruc}
\begin{aligned}
\mbox{For any $\sigma<\sup_{\Gamma}f$ and } \lambda\in \Gamma \mbox{  we have } \lim_{t\rightarrow +\infty}f(t\lambda)>\sigma,
\end{aligned}
\end{equation}
 and the unbounded condition
 \begin{equation}\label{unbounded}\begin{aligned}\,& \lim_{t\rightarrow+\infty}f(\lambda_1,\cdots,\lambda_{n-1}, \lambda_n+ t)=
 \sup_\Gamma f, \,& \forall \lambda=(\lambda_1,\cdots,\lambda_n)\in\Gamma.
\end{aligned}\end{equation}
Conditions \eqref{elliptic} and \eqref{concave} coincide respectively the ellipticity and concavity of equation \eqref{mainequ}  for  solutions $u$ in the class of $C^2$-\textit{admissible} functions
 pointwise satisfying  $\lambda(\mathfrak{g}[u])\in \Gamma$.
Also the constant $$\delta_{\psi,f}:= \inf_{M} \psi -\sup_{\partial \Gamma} f, \mbox{ where } \sup_{\partial \Gamma}f =\sup_{\lambda_{0}\in \partial \Gamma } \limsup_{\lambda\rightarrow \lambda_{0}}f(\lambda)$$ is used to measure whether  or not the equation is degenerate.
 More explicitly,  \eqref{mainequ}
 is called non-degenerate if the right-hand side satisfies 
\begin{equation}
 \label{nondegenerate}
  \inf_{M} \psi >\sup_{\partial \Gamma} f, 
 \end{equation}
while it is called degenerate if $\inf_M \psi=\sup_{\partial\Gamma}f$.
 
 The Dirichlet problem was studied by Caffarelli-Kohn-Nirenberg-Spruck \cite{CKNS2} for complex Monge-Amp\`ere
equation ($\chi\equiv0$) on bounded strictly pseudoconvex domains  
$\Omega\subset\mathbb{C}^n$,  
later extended by Guan \cite{Guan1998The} to general bounded domains by replacing 
  strictly pseudoconvex restriction to boundary  by a subsolution assumption 
  satisfying 
  \begin{equation}\label{existenceofsubsolution}\begin{aligned} f(\lambda(\mathfrak{g}[\underline{u}])) \geq  \psi,  \mbox{  } \lambda(\mathfrak{g}[\underline{u}])\in\Gamma \mbox{ in } \bar M,  \mbox{  and  }\underline{u}=  \varphi  \mbox{ on } \partial  M. \end{aligned}\end{equation}
The  subsolution is imposed as 
 a vital tool to deal with second order boundary estimates, as done by \cite{Guan1993Boundary,Guan1998The,Hoffman1992Boundary} for Dirichlet problem of Mong-Amp\`ere type equation on bounded domains;
 in addition, subsolution
  has a great advantage in the application to 
 certain geometric problems as it relaxes restrictions to the shape of boundary; see e.g. \cite{Chen,GuanP2002The,Guan2009Zhang}. Under the subsolution assumption, Li \cite{LiSY2004} studied Dirichlet problem \eqref{mainequ} for a class of equations 
  ($\chi\equiv0$) on bounded domains, which was extended by the author \cite{yuan2021cvpde} to K\"ahler manifolds with nonnegative orthogonal
bisectional curvature.
On general complex manifolds without imposing such curvature assumption, the Dirichlet problem
 has only been solved in rather restricted cases, among others include  complex Monge-Amp\`ere equation  \cite{Boucksom2012,Guan2010Li},   
complex inverse $\sigma_k$ equations \cite{Guan2015Sun}, complex $k$-Hessian equations \cite{Collins2019Picard}, 
and general equations satisfying \eqref{unbounded} and $\Gamma=\Gamma_n$   \cite{yuan2021cvpde}.
However,  little is known for more general fully nonlinear elliptic equations. 
 The primary obstruction is to prove gradient estimate, which is, however, highly open on general complex manifolds.
 It is pretty hard to prove gradient bound directly, as B{\l}ocki \cite{Blocki09gradient} and Guan-Li \cite{Guan2010Li} did for 
 complex Monge-Amp\`ere equation.
Blow-up argument is an alternative approach to deriving gradient estimate, as shown
 by Dinew-Ko{\l}odziej for complex $k$-Hessian equations on closed K\"ahler manifolds by combining Liouville type theorem   \cite[Theorem 0.1]{Dinew2017Kolo} with Hou-Ma-Wu's
   estimate \cite[Theorem 1.1]{HouMaWu2010} 
\begin{equation}
\label{sec-estimate-quar1}
\begin{aligned}
\sup_{M}\Delta u\leq C (1+\sup_{M} |\nabla u|^2).
\end{aligned}
\end{equation}
Recently, Hou-Ma-Wu's estimate and Dinew-Ko{\l}odziej's Liouville type theorem have been extended extensively by Sz\'ekelyhidi \cite[Proposition 13, Theorem 20]{Gabor}  to very general cases.
 In an attempt to solve the Dirichlet problem, one needs to prove a quantitative version of second order boundary estimates
  \begin{equation}
  \label{bdy-sec-estimate-quar1}
\begin{aligned}
\sup_{\partial M} \Delta u \leq C(1+\sup_M |\nabla u|^2),
\end{aligned}
\end{equation}
 as shown by Chen \cite{Chen} and complemented by 
 \cite{Boucksom2012,Phong-Sturm2010} 
  for complex Monge-Amp\`ere equation; while their proof relies heavily on the specific structure of Monge-Amp\`ere operator, 
thus it does not adapt to general equations.
 
 The author has made some progress in this direction.
  In \cite{yuan2017} the author  derived quantitative boundary estimate \eqref{bdy-sec-estimate-quar1} for Dirichlet problem \eqref{mainequ} for general equations on Hermitian manifolds with  holomorphically flat boundary, 
 which was further extended by the author \cite{yuan2019} to more general case when  
 the Levi form of $\partial M$, denoted by $L_{\partial M}$, satisfies for any $z\in\partial M$
  \begin{equation}
  \label{bdry-assumption1}
\begin{aligned}
(-\kappa_1,\cdots, -\kappa_{n-1})\in \overline{\Gamma}_\infty,
\end{aligned}
\end{equation}
  where and hereafter $\kappa_1,\cdots, \kappa_{n-1}$ denote the eigenvalues of $L_{\partial M}$ with respect to 
  $\omega'=\omega|_{T_{\partial M}\cap JT_{\partial M}}$, and $\overline{\Gamma}_\infty$ is the closure of $\Gamma_\infty$.
  Hereafter  $$\Gamma_\infty:=\{(\lambda_1,\cdots,\lambda_{n-1}): (\lambda_1,\cdots,\lambda_{n-1}, \lambda_n)\in \Gamma\}$$
  is the projection of $\Gamma$ into $\mathbb{R}^{n-1}$. 
  Such an assumption on shape of boundary is used to compare $\mathfrak{g}_{\alpha\bar\beta}$ with  
  $\underline{\mathfrak{g}}_{\alpha\bar\beta}$ when restricted to boundary, which enables us to apply Lemmas \ref{yuan's-quantitative-lemma} and \ref{asymptoticcone1} to understand the quantitative version of boundary estimate for double normal derivatives. A follow-up work was presented later in \cite{yuan2019PAMQ}, where the author studied equations on K\"ahler cones  from Sasaki geometry; see  \cite{Guan2009Zhang,Qiu2017Yuan} for more references concerning such equations.

The purpose of this paper is to drop the restriction to boundary.
To compare $\mathfrak{g}_{\alpha\bar\beta}$ with  
  $\underline{\mathfrak{g}}_{\alpha\bar\beta}$ on boundary, we use a method of Caffarelli-Nirenberg-Spruck \cite{CNS3}.
We recall briefly their method of deriving boundary estimate for Dirichlet problem 
\begin{equation}  \label{equation-cns} \begin{aligned}
\,& f(\lambda(D^2 u))=\psi \mbox{ in } \Omega\subset\mathbb{R}^n, \,& u=\varphi \mbox{ on } \partial \Omega. \nonumber
 \end{aligned} \end{equation}
 In order to compare $u_{\alpha\beta}$
 with certain data on boundary,  
they have constructed delicate barrier functions based on a characterization of $\Gamma_\infty$ then dealt with boundary estimates for double normal derivatives in the unbounded case; while their estimate is not quantitative and does not figure out how does it rely on the gradient bound.  
We also refer to \cite{LiSY2004} for complex equations in $\mathbb{C}^n$ and 
to \cite{Trudinger95} for the boundary estimate in bounded case. 

Combining the method from  \cite{CNS3} with the idea of \cite{yuan2017,yuan2019},
we prove the quantitative boundary estimate without restriction to boundary. 
 This is new even when $M$ is a  bounded domain in $\mathbb{C}^n$.

\begin{theorem}
\label{thm0-estimate}
 Let $(M,J,\omega)$ be a compact Hermitian manifold  with smooth boundary, $\psi$, $\varphi\in C^\infty$. 
Suppose \eqref{elliptic}, \eqref{concave},  \eqref{addistruc}, \eqref{unbounded}, \eqref{nondegenerate}  and  \eqref{existenceofsubsolution} hold. 
For any admissible solution $u\in C^3(M)\cap C^{2}(\bar M)$ to the Dirichlet problem \eqref{mainequ},
 we have quantitative boundary estimate \eqref{bdy-sec-estimate-quar1}.
\end{theorem}

As a consequence of Theorem \ref{thm0-estimate}, together with the blow-up argument used in \cite{Gabor,Chen}, we completely solve the Dirichlet problem for a large class of fully nonlinear elliptic equations on general Hermitian manifolds.
  
\begin{theorem}
\label{thm1-diri}
  Let $(M,J,\omega)$ and $(f,\Gamma)$ be as in Theorem \ref{thm0-estimate}.
  Then for $\varphi\in C^{\infty}(\partial M)$ and $\psi\in C^{\infty}(\bar M)$ with  \eqref{nondegenerate}, the 
 Dirichlet problem \eqref{mainequ} 
has a unique smooth admissible solution.  
\end{theorem}

A basic work in K\"ahler geometry is Yau's \cite{Yau78} proof of Calabi's conjecture.
Recently, 
Sz\'ekelyhidi-Tosatti-Weinkove \cite{STW17} solved Gauduchon's conjecture \cite{Gauduchon84} for dimension $n\geq3$, thereby extending the Calabi-Yau theorem to non-K\"ahler geometry. The Gauduchon conjecture is reduced to solving a Monge-Amp\`ere equation for $(n-1)$-plurisubharmonic  ($(n-1)$-PSH for short)
functions
  \begin{equation}
 \label{MA-n-1}
\begin{aligned}
\left(\tilde{\chi}+\frac{1}{n-1}\left(\Delta u\omega-\sqrt{-1}\partial \overline{\partial}u\right)+Z\right)^n= e^\phi \omega^n  
\end{aligned}
\end{equation} 
on a closed $n$-complex dimensional Hermitian manifold $(M,\omega)$, 
where 
\[\tilde{\chi}=\frac{1}{(n-1)!}* \omega_0^{n-1}, \mbox{  } Z = \frac{1}{(n-1)!}*\mathfrak{Re}(\sqrt{-1}\partial u\wedge \overline{\partial}(\omega^{n-2})),\]
 here $\omega_0$ is a Gauduchon metric, $*$ is the Hodge star operator with respect to $\omega$; see  \cite{Popovici2015,TW19}. 
   Following \cite{Harvey2012Lawson}, also \cite{TW17}, we call $u$ is $(n-1)$-PSH if 
\begin{equation}
\label{gamma-cone}
\begin{aligned}
\tilde{\chi}+\frac{1}{n-1}\left(\Delta u\omega-\sqrt{-1}\partial \overline{\partial}u\right)+Z>0 \mbox{ in }  M.
\end{aligned}
 \end{equation}
 When $n=2$ the equation is a standard complex Monge-Amp\`ere equation that was solved by Cherrier \cite{Cherrier1987}. 
 We refer to \cite{FuWangWuFormtype2010,FuWangWuFormtype2015} for related topics.
   
    The Dirichlet problem for equation
 \eqref{MA-n-1} possibly with degenerate right-hand side was solved by the author in the second part of \cite{yuan2019}, in which
 the boundary is \textit{mean pseudoconcave} in the sense that 
 \begin{equation}
 \label{mean-pseudoconcave1}
\begin{aligned}
-(\kappa_1+\cdots+\kappa_{n-1})\geq 0 \mbox{ on } \partial M.
\end{aligned}
\end{equation} 
 
 In the second part of this paper, we drop such an assumption and completely solve the Dirichlet problem.
 
 \begin{theorem}
 \label{thm0-n-1-yuan3}
 Let $(M,J,\omega)$ be a compact Hermitian manifold with smooth boundary.  Assume the given data $\varphi$, $\phi$ are smooth.
 Suppose there is a $C^{2,1}$-smooth $(n-1)$-PSH function $\underline{u}$ such that 
  \begin{equation}
  \label{subsolution-MA-n-1}
\begin{aligned}
\left(\tilde{\chi}+\frac{1}{n-1}\left(\Delta \underline{u}\omega-\sqrt{-1}\partial \overline{\partial}\underline{u}\right)+\underline{Z}\right)^n\geq e^\phi \omega^n \mbox{ in } M, \mbox{  } \underline{u}=\varphi \mbox{ on } \partial M, 
\end{aligned}
\end{equation} 
where $\underline{Z} = \frac{1}{(n-1)!}*\mathfrak{Re}(\sqrt{-1}\partial \underline{u}\wedge \overline{\partial}(\omega^{n-2}))$. Then 
 the Dirichlet problem for 
 equation \eqref{MA-n-1} with boundary data
\begin{equation}
\label{bdy-value-2}
\begin{aligned}
u=\varphi \mbox{ on } {\partial M}
\end{aligned}
\end{equation}
 is uniquely solvable in class of smooth $(n-1)$-PSH functions.
 \end{theorem}

\vspace{1.5mm} 
The paper is organized as follows. In Section \ref{sec2} we briefly sketch the proof. In Sections \ref{sec3} and \ref{sec4} the  
quantitative boundary estimates for Dirichlet problem \eqref{mainequ} and \eqref{MA-n-1}-\eqref{bdy-value-2} are proved respectively. In Appendix \ref{appendix1} we append two key lemmas which are key ingredients.

 \section{Sketch of proof}
 \label{sec2}

For equations \eqref{mainequ} and \eqref{MA-n-1} the following second order estimate 
\begin{equation}
\label{quantitative-2nd-boundary-estimate}
\begin{aligned}
\sup_{ M} \Delta u
\leq C(1+ \sup_{M}|\nabla u|^{2} +\sup_{\partial M}|\Delta u|)
\end{aligned}
\end{equation}
was proved by Sz\'ekelyhidi \cite[Proposotion 13]{Gabor} and Sz\'ekelyhidi-Tosatti-Weinkove 
   \cite[Section 3]{STW17} respectively.

Our goal is to derive the quantitative boundary estimate \eqref{bdy-sec-estimate-quar1}, i.e.,
  \begin{equation}  
   \begin{aligned}
\sup_{\partial M} \Delta u \leq C(1+\sup_M |\nabla u|^2). \nonumber
\end{aligned}\end{equation}
In \cite{yuan2019}, 
the author proved quantitative boundary estimate for tangential-normal derivatives.

\begin{proposition}
[\cite{yuan2019}]
\label{mix-general}
  Let $(M,\omega)$ be a compact Hermitian manifold 
 with $C^3$-smooth boundary, $\varphi\in C^3(\partial M)$, $\psi\in C^1(\bar M)$.
Suppose \eqref{elliptic}, \eqref{concave}, \eqref{nondegenerate} and \eqref{existenceofsubsolution} hold.
Then 
for any admissible solution $u\in C^3(M)\cap C^2(\bar M)$ to the Dirichlet problem \eqref{mainequ},
 there is a uniform positive constant $C$ depending on $|\varphi|_{C^{3}(\bar M)}$, 
 $|\underline{u}|_{C^{2}(\bar M)}$,  $|u|_{C^0(M)}$, $|\nabla u|_{C^0(\partial  M)}$,
$\sup_{M}|\nabla\psi|$, $\partial M$ 
up to third derivatives
and other known data (but neither on $(\delta_{\psi,f})^{-1}$ nor on $\sup_{M}|\nabla u|$)
such that
\begin{equation}
\label{quanti-mix-derivative-00} 
|\nabla^2 u(X,\nu)|\leq C(1+\sup_{M}|\nabla u|)
 \end{equation}
 for any $X\in T_{\partial M}$ with $ |X|=1$,   
where  $\nabla^2 u$ denotes the real Hessian of $u$, and $\nu$ denotes the unit inner normal vector along the boundary.
\end{proposition}

\begin{proposition}
[\cite{yuan2019}]
\label{mix-general-2}
 
 Let $(M,\omega)$ be a compact Hermitian manifold with $C^3$ boundary, $\varphi\in C^3(\partial M)$, $\psi\in C^1(\bar M)$. 
Let $u\in C^3(M)\cap C^2(\bar M)$ be a $(n-1)$-PSH function to solve the Dirichlet problem
 \eqref{MA-n-1} with boundary value condition \eqref{bdy-value-2}. Suppose \eqref{subsolution-MA-n-1} holds. 
 Then the estimate \eqref{quanti-mix-derivative-00} 
holds for a uniform positive constant $C$ depending on $|\varphi|_{C^{3}(\bar M)}$, 
 $|\underline{u}|_{C^{2}(\bar M)}$,  $|u|_{C^0(M)}$, $|\nabla u|_{C^0(\partial  M)}$,
$\sup_{M}|\nabla\psi|$, $\partial M$ 
up to third derivatives
and other known data. 
\end{proposition}

To complete the proof of Theorems \ref{thm0-estimate} and \ref{thm0-n-1-yuan3}, as in \cite{yuan2017,yuan2019}, we are guided toward 

\begin{proposition}
\label{proposition-quar-yuan1}
Let $(M, J,\omega)$ be a compact Hermitian manifold with $C^3$ boundary.  
Let $u\in  C^{3}(\bar M)$ be an \textit{admissible} solution
 to Dirichlet problem \eqref{mainequ} with data $\psi\in C^0(\bar M)$ and $\varphi\in C^2(\partial M)$. 
We denote
\begin{equation}
\begin{aligned}
\,& T^{1,0}_{\partial M}=T^{1,0}_{\bar M} \cap T^{\mathbb{C}}_{ \partial M},   
\,& \xi_n=\frac{1}{\sqrt{2}}(\mathrm{{\bf \nu}}-\sqrt{-1}J\nu). \nonumber
\end{aligned}
\end{equation}
Suppose  \eqref{elliptic}, \eqref{concave},  \eqref{addistruc}, \eqref{unbounded}, \eqref{nondegenerate} and \eqref{existenceofsubsolution}
   hold.
Then for   $\xi_\alpha$, $\xi_\beta \in T^{1,0}_{\partial M, x_0}$, $x_0\in \partial M$, $1\leq\alpha,\beta\leq n-1$ satisfying
  $g(\xi_\alpha, \bar\xi_{\beta}) =\delta_{\alpha\beta}$ at $x_0$,
we have
\begin{equation}
\label{yuan-prop1}
\begin{aligned}
 \mathfrak{g}(\xi_n, J\bar \xi_n)(x_0) 
  \leq C\left(1 +  \sum_{\alpha=1}^{n-1} |\mathfrak{g}(\xi_\alpha, J\bar \xi_n)(x_0)|^2\right), 
\end{aligned}
\end{equation}
where $C$ is a uniform positive constant  depending  only on   
$(\delta_{\psi,f})^{-1}$, 
 $|u|_{C^0(M)}$,   $|\nabla u|_{C^0(\partial M)}$, $|\underline{u}|_{C^{2}(\bar M)}$, 
  $\partial M$ up to third  derivatives 
 and other known data
 (but not on $\sup_{M}|\nabla u|$).
\end{proposition}

\begin{proposition}
\label{proposition-quar-yuan2}
Let $(M, J,\omega)$ be a compact Hermitian manifold with $C^3$ boundary, let 
$\psi\in C^0(\bar M)$ and $\varphi\in C^2(\partial M)$.  
In addition we assume \eqref{subsolution-MA-n-1} is satisfied.
For any  $(n-1)$-PSH function  $u\in  C^{2}(\bar M)$ solving
 Dirichlet problem \eqref{MA-n-1} and \eqref{bdy-value-2}, 
then there is a uniform positive constant $C$ depending  only on    $(\delta_{\psi,f})^{-1}$, 
 $|u|_{C^0(M)}$,   $|\nabla u|_{C^0(\partial M)}$, 
 $|\underline{u}|_{C^{2}(\bar M)}$,  $\partial M$ up to third derivatives 
 and other known data, such that for any $x_0\in \partial M$, 
\begin{equation}
\label{yuan-prop2}
\begin{aligned}
 \mathfrak{\tilde{g}}(\xi_n, J\bar \xi_n)(x_0) 
  \leq C\left(1 +  \sum_{\alpha=1}^{n-1} |\mathfrak{\tilde{g}}(\xi_\alpha, J\bar \xi_n)(x_0)|^2\right).
\end{aligned}
\end{equation}
  Here $ \mathfrak{\tilde{g}}$ is denoted by \eqref{yuan3-buchong111}, and
   $\xi_i$ are as in Proposition \ref{proposition-quar-yuan1}.
\end{proposition}


  \section{Proof of Proposition \ref{proposition-quar-yuan1}}
  \label{sec3}
  
  We always assume $\Gamma$ is of type 1 in the sense of \cite{CNS3}, then $\Gamma_\infty$ is an open symmetric convex cone in $\mathbb{R}^{n-1}$;
 otherwise we have done by \cite{yuan2019} as $\Gamma_\infty=\mathbb{R}^{n-1}$ whenever $\Gamma$ is of type 2.
 
 In this section we always denote $\mathfrak{g}=\mathfrak{g}[u]$, $\underline{\mathfrak{g}}={\mathfrak{g}}[\underline{u}]$, 
$ \mathfrak{g}[v]=\sqrt{-1}\mathfrak{g}[v]_{i\bar j}dz_i\wedge d\bar z_j$, $ \chi=\sqrt{-1}\chi_{i\bar j}dz_i\wedge d\bar z_j$, 
$v_i=\frac{\partial v}{\partial z_i}$, $v_{i\bar j}=\frac{\partial^2 v}{\partial z_i \partial \bar z_j}$, etc. 
And the Greek letters, such as $\alpha, \beta$, range from $1$ to $n-1$.
 We also denote $\sigma(z)$ by the distance function from $z$ to $\partial M$ with respect to $\omega$. 
     
    \subsection{First ingredient of proof}
  

  Given $p_0\in\partial M$. We can choose a local holomorphic coordinate systems $(z_1,\cdots,z_n)$, $z_i=x_i+\sqrt{-1}y_i$, centered at $p_0$, such that $g_{i\bar j}(0)=\delta_{ij}$, $\frac{\partial}{\partial x_n}$ is the inner normal vector at the origin, and
   $T^{1,0}_{{p_0},{\partial M}}$ is spanned by $\frac{\partial}{\partial z_\alpha}$ for $1\leq\alpha\leq n-1$.
    Let' denote
 \begin{equation}
A(R)=\left(
\begin{matrix}
\mathfrak{g}_{1\bar 1}&\mathfrak{g}_{1\bar 2}&\cdots &\mathfrak{g}_{1\overline{(n-1)}} &\mathfrak{g}_{1\bar n}\\
\mathfrak{g}_{2\bar 1} &\mathfrak{g}_{2\bar 2}&\cdots& \mathfrak{g}_{2\overline{(n-1)}}&\mathfrak{g}_{2\bar n}\\
\vdots&\vdots&\ddots&\vdots&\vdots \\
\mathfrak{g}_{(n-1)\bar 1}&\mathfrak{g}_{(n-1)\bar 2}& \cdots&  \mathfrak{g}_{{(n-1)}\overline{(n-1)}}& \mathfrak{g}_{(n-1)\bar n}\\
\mathfrak{g}_{n\bar 1}&\mathfrak{g}_{n\bar 2}&\cdots& \mathfrak{g}_{n \overline{(n-1)}}& R  \nonumber
\end{matrix}
\right).
\end{equation}
 Denote by $\lambda'=\lambda_{\omega'}(\mathfrak{g}_{\alpha\bar\beta})$, 
 $\underline{\lambda}'=\lambda_{\omega'}(\underline{\mathfrak{g}}_{\alpha\bar\beta})$. We know
    \begin{equation}
    \label{yuan3-31}
    \begin{aligned}
    \lambda', \mbox{  } \underline{\lambda}' \in \Gamma_\infty.
    \end{aligned}
    \end{equation}
By the openness of $\Gamma$ we know there is a uniform positive constant $\varepsilon_0$ depending only on 
$\inf_{z\in\partial M}|\mathfrak{\underline{g}}(z)|$ such that
\begin{equation}
\label{key-002-yuan3}
\begin{aligned}
(\underline{\lambda}'_1-\varepsilon_{0}, \cdots, \underline{\lambda}'_{n-1}-\varepsilon_0, R_0)\in \Gamma.
\end{aligned}
\end{equation}

The boundary value condition implies 
  \begin{equation}
  \label{yuan3-buchong5}
    \begin{aligned}
  \,&  u_\alpha(0)=\underline{u}_\alpha(0), \,&
    u_{\alpha\bar\beta}(0)=\underline{u}_{\alpha\bar\beta}(0)+(u-\underline{u})_{x_n}(0)\sigma_{\alpha\bar\beta}(0).
    \end{aligned}
    \end{equation}
    Let $\eta=(u-\underline{u})_{x_n}(0)$, thus at $p_0$ ($z=0$)
     \begin{equation}
    \begin{aligned}
        \mathfrak{g}_{\alpha\bar\beta}=\underline{\mathfrak{g}}_{\alpha\bar\beta}+ \eta\sigma_{\alpha\bar\beta}.
    \end{aligned}
    \end{equation}
    By the maximum principle and boundary value condition,  one derives
    \begin{equation}
     \label{key-14-yuan3}
     \begin{aligned}
   \,& 0\leq \eta\leq (\check{u}-\underline{u})_{x_n}(0),  \,& \underline{u}\leq u\leq \check{u} \mbox{ in } M,
     \end{aligned}
\end{equation}
where $\check{u}$ is a supersolution solving
 \begin{equation}
 \label{supersolution-1}
     \begin{aligned}
  \,&   \mathrm{tr}_\omega(\mathfrak{g}[\check{u}])=0 \mbox{ in } M,   \,& \check{u}=\varphi \mbox{ on } \partial M.
 \end{aligned}
\end{equation}
The existence of $\check{u}$ follows from standard theory of elliptic equations of second order.

    When $\eta=0$, one has $u_{\alpha\bar\beta}=\underline{u}_{\alpha\bar\beta}$  and the proof is the same as that from \cite{yuan2017} where $\partial M$ is Levi flat; see also \cite{yuan2019} for the case when $\partial M$ satisfies \eqref{bdry-assumption1}.
   
    From now on we assume $\eta>0$.
         We rewrite $u_{\alpha\bar\beta}+\chi_{\alpha\bar\beta}$ as 
    \begin{equation}
    \begin{aligned}
        u_{\alpha\bar\beta}+\chi_{\alpha\bar\beta}= (1-t)(\underline{u}_{\alpha\bar\beta}+\chi_{\alpha\bar\beta})
        +\left(t(\underline{u}_{\alpha\bar\beta}+\chi_{\alpha\bar\beta})+\eta\sigma_{\alpha\bar\beta}\right).
    \end{aligned}
    \end{equation}
For simplicity, as in \cite{CNS3} also as in \cite{LiSY2004},  we denote
     \begin{equation}
    \begin{aligned}
    A_t=\sqrt{-1}[t(\underline{u}_{\alpha\bar\beta}+\chi_{\alpha\bar\beta})+\eta\sigma_{\alpha\bar\beta}]dz_\alpha\wedge d\bar z_\beta.
    \end{aligned}
    \end{equation}
    Clearly, $(A_{1})_{\alpha\bar\beta}={u}_{\alpha\bar\beta}+\chi_{\alpha\bar\beta}$ so  $\lambda_{\omega'}(A_1)\in\Gamma_\infty$. 
 On the other hand,   $\lambda_{\omega'}(A_t)\in \mathbb{R}^{n-1}\setminus\Gamma_\infty$ for $t\ll -1$.
 Let $t_0$ be the first 
  $t$ as we decrease $t$ from $1$ 
  so that
      \begin{equation}
      \label{key0-yuan3}
    \begin{aligned}
    \lambda_{\omega'}(A_{t_0})\in\partial\Gamma_\infty.
    \end{aligned}
    \end{equation}
  Clearly, for a uniform positive constant $T_0$ under control,
    \begin{equation}
    \label{1-yuan3}
    \begin{aligned}
   -T_0< t_0<1.
    \end{aligned}
    \end{equation}
    We then have the following identity
   \begin{equation}
 {A}(R)=\left(
\begin{matrix}
(1-t_0)(\mathfrak{\underline{g}}_{\alpha\bar \beta}-\frac{\varepsilon_0}{4} \delta_{\alpha\beta}) &\mathfrak{g}_{\alpha\bar n}\\
\mathfrak{g}_{n\bar \beta}& R/2  \nonumber
\end{matrix}
\right) + \left(
\begin{matrix}
(A_{t_0})_{\alpha\bar \beta}+\frac{\varepsilon_0(1-t_0)}{4}  \delta_{\alpha\beta} &0\\
0& R/2  \nonumber
\end{matrix}
\right),
\end{equation}
here $\varepsilon_0$ is the constant from \eqref{key-002-yuan3}.
Let  $A'(R)=\left(
\begin{matrix}
(1-t_0)(\mathfrak{\underline{g}}_{\alpha\bar \beta}-\frac{\varepsilon_0}{4}\delta_{\alpha\beta}) &\mathfrak{g}_{\alpha\bar n}\\
\mathfrak{g}_{n\bar \beta}& R/2  \nonumber
\end{matrix}
\right),
$
and
$A''(R)=\left(
\begin{matrix}
(A_{t_0})_{\alpha\bar \beta}+\frac{(1-t_0)\varepsilon_0}{4} \delta_{\alpha\beta} &0\\
0& R/2  
\end{matrix}
\right).$
 Denote
     \begin{equation}
     \label{yuan3-buchong3}
     \tilde{\lambda}'=(\tilde{\lambda}_1',\cdots,\tilde{\lambda}_{n-1}'):=\lambda_{\omega'}(A_{t_0}).
     \end{equation}
By \eqref{1-yuan3} and \eqref{key-14-yuan3}, there is a uniform constant $C_0>0$ so that  $|\tilde{\lambda}'|\leq C_0$, that is
 $\tilde{\lambda'}$ is contained a compact subset of $\partial\Gamma_\infty$, i.e.  
 \begin{equation}
 \label{yuan3-buchong4}
\begin{aligned}
\tilde{\lambda}'\in {K}':=\{\lambda'\in\partial\Gamma_\infty: |\lambda'|\leq C_0\}.
\end{aligned}
\end{equation}
   So there is a 
    uniform positive constant $R'$ possibly depending on $((1-t_0)\varepsilon_0)^{-1}$ 
    such that for any $R>R'$,
    $$\lambda(A''(R))\in\Gamma.$$
    
    By the unbounded condition \eqref{unbounded}
    there is a uniform positive constant $R_1$ depending on $(1-t_0)^{-1}$, 
    $\varepsilon_0$ and $\underline{\lambda}'$
    such that
     \begin{equation}
     \label{key-03-yuan3}
\begin{aligned}
f((1-t_0)(\underline{\lambda}'_{1}-\frac{\varepsilon_0}{2}),\cdots, (1-t_0)(\underline{\lambda}'_{n-1}
-\frac{\varepsilon_0}{2}), R_1)>f(\lambda(\underline{\mathfrak{g}})).
\end{aligned}
\end{equation}
Here is the only place where we use the unbounded condition \eqref{unbounded}. As a contrast,  such a unbounded condition  can be removed when $t_0\leq0$ occurs.
    
Let's pick  $\epsilon=\frac{(1-t_0)\varepsilon_0}{4}$ in 
 Lemma  \ref{yuan's-quantitative-lemma}, and we set
\begin{equation}
\begin{aligned}
 \frac{R_c}{2}=\,& \frac{4(2n-3)}{(1-t_0)\varepsilon_0}
\sum_{\alpha=1}^{n-1} | \mathfrak{g}_{\alpha\bar n}|^2
 + (n-1)\sum_{\alpha=1}^{n-1}  | \mathfrak{\underline{g}}|
 +R_0+R_1+R'
 \\
 \,&+\frac{(n-1)(1-t_0)\varepsilon_0}{4}
    +\frac{(n-2)(1-t_0)\varepsilon_0}{4(2n-3)}, \nonumber
\end{aligned}
\end{equation}
where $\varepsilon_0$, $R_0$, $R_1$ and $R'$ are fixed constants we have chosen above.
It follows from Lemma   \ref{yuan's-quantitative-lemma} 
   that the eigenvalues $\lambda({A}'(R_c))$ of ${A}'(R_c)$
   (possibly with an order) shall behavior like
\begin{equation}
\label{lemma12-yuan}
\begin{aligned}
\,& \lambda_{\alpha}({A}'(R_c)\geq (1-t_0)(\underline{\lambda}'_{\alpha}-\frac{\varepsilon_0}{2}), \mbox{  } 1\leq \alpha\leq n-1, \\\,&
\lambda_{n}({A}'(R_c)\geq R_c/2-(n-1)(1-t_0)\varepsilon_0/4.
\end{aligned}
\end{equation}
In particular, $\lambda({A}'(R_c))\in \Gamma$. So $\lambda(A(R_c))\in \Gamma$.
Note that 
 \eqref{concavity2} implies 
\begin{equation}
\label{yuan-k1}
    \begin{aligned}
    f(\lambda(A(R_c)))\geq f(\lambda(A'(R_c))). 
    \end{aligned}
    \end{equation}
   It follows from \eqref{key-03-yuan3}, \eqref{lemma12-yuan},   \eqref{yuan-k1} that
    \begin{equation}
     \label{yuan-k2}
\begin{aligned}
\mathfrak{g}_{n\bar n} \leq R_c.  
\end{aligned}
\end{equation}

\subsection{Second ingredient of proof}

    To complete the proof of Proposition \ref{proposition-quar-yuan1}, from \eqref{yuan-k2}, it requires only to prove $(1-t_0)^{-1}$ can be uniformly bounded from above. In other words
    \begin{equation}
    \label{bound-t0}
    \begin{aligned}
   (1-t_0)^{-1}\leq C. 
    \end{aligned}
    \end{equation}
    
      In the case when $\partial  M$ is holomorphically flat or more generally \eqref{bdry-assumption1}, one can check $t_0\leq 0$ as shown in \cite{yuan2017,yuan2019} previously  where \eqref{unbounded} is not necessary.  For general case without restriction to boundary, we prove
    \begin{lemma}
    \label{keylemma1-yuan3}
    Let $t_0$ be as defined above, then the inequality \eqref{bound-t0} holds for a uniform positive constant $C$ depending on 
    $|u|_{C^0(M)}$, $|\nabla u|_{C^0(\partial M)}$, $|\underline{u}|_{C^2(M)}$, 
    $(\delta_{\psi,f})^{-1}$, $\partial M$ up to third derivatives and other known data.
      
\end{lemma}

We follow some idea of Caffarelli-Nirenberg-Spruck \cite{CNS3}, further extended by Li \cite{LiSY2004} 
to equations  in $\mathbb{C}^n$,
 and use some notation  of \cite{CNS3,LiSY2004}. 
  
   Without loss of generality $t_0>\frac{1}{2}$ and $\tilde{\lambda}_1'\leq \cdots \leq \tilde{\lambda}_{n-1}'$ (as denoted in \eqref{yuan3-buchong3}).
    It was proved in \cite[Lemma 6.1]{CNS3} that for $\tilde{\lambda}'\in\partial\Gamma_\infty$ there is a supporting plane for
     $\Gamma_\infty$ and one can choose $\mu_j$ with 
   $\mu_1\geq \cdots\geq \mu_{n-1}\geq0$ so that
    \begin{equation}
    \label{key-18-yuan3}
    \begin{aligned}
    \Gamma_\infty\subset \left\{\lambda'\in\mathbb{R}^{n-1}: \sum_{\alpha=1}^{n-1}\mu_\alpha\lambda'_\alpha>0 \right\}, \mbox{  }
  \mbox{  } \sum_{\alpha=1}^{n-1} \mu_\alpha=1, \mbox{  } \sum_{\alpha=1}^{n-1}\mu_\alpha \tilde{\lambda}_\alpha'=0.
    \end{aligned}
    \end{equation} 
  By a result of \cite{Marcus1956} (see also  \cite[Lemma 6.2]{CNS3}) (assume $\underline{\lambda}_1'\leq \cdots\leq\underline{\lambda}_{n-1}'$)
   \begin{equation}
    \begin{aligned}
    \sum_{\alpha=1}^{n-1} \mu_\alpha \underline{\mathfrak{g}}_{\alpha\bar\alpha}\geq \sum_{\alpha=1}^{n-1}\mu_\alpha \underline{\lambda}'_\alpha\geq \inf_{p\in\partial M}\sum_{\alpha=1}^{n-1} \mu_\alpha \underline{\lambda}'_\alpha(p)\geq a_0>0.
    \end{aligned}
    \end{equation}
    Here we use \eqref{yuan3-31}, \eqref{yuan3-buchong4} and \eqref{key-18-yuan3}. We shall mention that $a_0$ depends on $\mathrm{disc}(\lambda(\mathfrak{\underline{g}}),\partial \Gamma)$. 
   Without loss of generality,  we assume $({A_{t_0}})_{\alpha\bar\beta}=t_0\underline{\mathfrak{g}}_{\alpha\bar\beta}+\eta\sigma_{\alpha\bar\beta}$ is diagonal at $p_0$. From \eqref{key-18-yuan3} one has at the origin
     \begin{equation}
    \begin{aligned}
  0=  t_0 \sum_{\alpha=1}^{n-1}\mu_\alpha\underline{\mathfrak{g}}_{\alpha\bar\alpha}+\eta\sum_{\alpha=1}^{n-1}\mu_\alpha  {\sigma}_{\alpha\bar\alpha} 
  \geq a_0t_0 +\eta \sum_{\alpha=1}^{n-1} \mu_\alpha \sigma_{\alpha\bar\alpha}.
    \end{aligned}
    \end{equation}
 Together with \eqref{key-14-yuan3}, we see at the origin $\{z=0\}$
    \begin{equation}
    \label{key-1-yuan3}
    \begin{aligned}
    -\sum_{\alpha=1}^{n-1} \mu_\alpha \sigma_{\alpha\bar\alpha}\geq \frac{a_0 t_0}{\sup_{\partial M}|\nabla (\check{u}-\underline{u})|}=:a_1>0,
    \end{aligned}
    \end{equation}
    where $\check{u}$ and $\underline{u}$ are respectively supersolution and subsolution.
         Let $$\Omega_\delta=M\cap B_{\delta}(0),$$
where $B_\delta(0)=\{z\in M: |z|<\delta\}$.
On $\Omega_\delta$, we let
\begin{equation}
    \begin{aligned}
    d(z)=\sigma(z)+\tau |z|^2
    \end{aligned}
    \end{equation}
    where $\tau$ is a positive constant 
     to be determined; and let 
    \begin{equation}
    \begin{aligned}
    w(z)=\underline{u}(z)+({\eta}/{t_0})\sigma(z)+l(z)\sigma(z)+Ad(z)^2,
    \end{aligned}
    \end{equation}
    where $l(z)=\sum_{i=1}^n (l_iz_i+\bar l_{i}\bar z_{i})$, 
    $l_i\in \mathbb{C}$, $\bar l_i=l_{\bar i}$,
     to be chosen as in \eqref{chosen-1}, and $A$ is a positive constant to be determined.
       Furthermore, on $\partial M\cap\bar\Omega_\delta$, $u(z)-w(z)=-A\tau^2|z|^4$.
    On $M\cap\partial B_{\delta}(0)$,
    \begin{equation}
    \begin{aligned}
    u(z)-w(z)=\,& u(z)-\underline{u}(z)-({\eta}/{t_0})\sigma(z)-l(z)\sigma(z)-Ad(z)^2 \\
    \leq \,& |u-\underline{u}|_{C^0(\Omega_\delta)} 
    -(2A\tau \delta^2+\frac{\eta}{t_0}-2n \sup_{i}|l_i| \delta)\sigma(z)-A\tau^2\delta^4 \\
    \leq \,& -\frac{A\tau^2 \delta^4}{2} \nonumber
    \end{aligned}
    \end{equation}
 provided $A\gg1$ depending on $|u|_{C^0(M)}$.

Let $T_1(z),\cdots, T_{n-1}(z)$ be an orthonormal basis for holomorphic tangent space  of level hypersurface
 $\{w: d(w)=d(z)\}$ at $z$, so that at the origin 
 $T_\alpha(0)= \frac{\partial }{\partial z_\alpha}$  for each $1\leq\alpha\leq n-1$.

  Such a basis exists: 
  We see at the origin  $\partial d(0)=\partial \sigma(0)$.  
  Thus for $1\leq \alpha\leq n-1$, we can choose $T_\alpha$ such that at the  origin
  $T_\alpha(0)= \frac{\partial }{\partial z_\alpha}$.

The result of \cite{Marcus1956} (see also  \cite[Lemma 6.2]{CNS3}) implies the following lemma.
\begin{lemma}
\label{lemma-yuan3-buchong1}
Let $T_1(z),\cdots, T_{n-1}(z)$ be as above, and  let  $T_n=\frac{\partial d}{|\partial d|}$.
For a real $(1,1)$-form $\Theta=\sqrt{-1}\Theta_{i\bar j}dz_i\wedge d\bar z_j$,
we denote by $\lambda(\Theta)=(\lambda_1(\Theta),\cdots,\lambda_n(\Theta))$  the eigenvalues of $\Theta$ (with respect to $\omega$) with order $\lambda_1(\Theta)\leq \cdots\leq \lambda_n(\Theta)$. Then for any 
$\mu_1\geq\cdots\geq\mu_n$,
\[\sum_{i=1}^n \mu_i \lambda_{i}(\Theta)\leq \sum_{i=1}^n\mu_i\Theta(T_i,J\bar T_i).\]
\end{lemma}

Let $\mu_1,\cdots,\mu_{n-1}$ be as in \eqref{key-18-yuan3}, and set $\mu_n=0$. Let's denote $T_\alpha=\sum_{k=1}^nT_\alpha^k\frac{\partial }{\partial z_k}$. 
For $\Theta=\sqrt{-1}\Theta_{i\bar j}dz_i\wedge d\bar z_j$, we define
\begin{equation}
    \begin{aligned}
    \Lambda_\mu(\Theta):= \sum_{\alpha=1}^{n-1}\mu_{\alpha} T_{\alpha}^i \bar T_{\alpha}^j \Theta_{i\bar j}.  \nonumber
    \end{aligned}
    \end{equation}
    
    \begin{lemma}
    \label{lemma-key2-yuan3}
   There are parameters $\tau$, $A$, $l_i$, $\delta$ depending only on $|u|_{C^0(M)}$, 
   $|\nabla u|_{C^0(\partial M)}$,  
   $|\underline{u}|_{C^2(M)}$, 
   $\partial M$ up to third derivatives and other known data, such that 
   \begin{equation}
    \begin{aligned}
 \,&    \Lambda_\mu (\mathfrak{g}[w])   \leq0  \mbox{ in } \Omega_\delta, \,& u\leq w \mbox{ on } \partial \Omega_\delta. \nonumber
    \end{aligned}
    \end{equation}
    \end{lemma}
    
    \begin{proof}
 By direct computation
     \begin{equation}
    \begin{aligned}
    \Lambda_\mu (\mathfrak{g}[w])
    =\,&\sum_{\alpha=1}^{n-1} \mu_\alpha T_{\alpha}^i \bar T_{\alpha}^j
     (\chi_{i\bar j}+\underline{u}_{i\bar j}+\frac{\eta}{t_0}\sigma_{i\bar j}) 
       + 2Ad(z)\sum_{\alpha=1}^{n-1} \mu_\alpha T_{\alpha}^i \bar T_{\alpha}^j d_{i\bar j}
       \\\,&
       + \sum_{\alpha=1}^{n-1} \mu_\alpha T_{\alpha}^i \bar T_{\alpha}^j (l(z)\sigma_{i\bar j}
       +l_i\sigma_{\bar j}+\sigma_i l_{\bar j}).  \nonumber
    \end{aligned}
    \end{equation}
\begin{itemize}
     \item
At the origin $\{z=0\}$, $T_{\alpha}^i=\delta_{\alpha i}$, 
\begin{equation}
    \begin{aligned}
   \sum_{\alpha=1}^{n-1} \mu_\alpha T_{\alpha}^i \bar T_{\alpha}^j
     (\chi_{i\bar j}+\underline{u}_{i\bar j}+\frac{\eta}{t_0}\sigma_{i\bar j}) (0)
     =\frac{1}{t_0}\sum_{\alpha=1}^{n-1}\mu_\alpha (A_{t_0})_{\alpha\bar\alpha}=0.  \nonumber
    \end{aligned}
    \end{equation}
    So there are complex constants $k_i$ 
    such that
    \begin{equation}
    \begin{aligned}
   \sum_{\alpha=1}^{n-1} \mu_\alpha T_{\alpha}^i \bar T_{\alpha}^j
     (\chi_{i\bar j}+\underline{u}_{i\bar j}+\frac{\eta}{t_0}\sigma_{i\bar j}) (z)=\sum_{i=1}^n (k_i z_i+ \bar k_{i} \bar z_{i})+O(|z|^2). \nonumber
    \end{aligned}
    \end{equation}

\item   
    \begin{equation}
    \begin{aligned}
    2Ad(z)\sum_{\alpha=1}^{n-1} \mu_\alpha T_{\alpha}^i \bar T_{\alpha}^j d_{i\bar j} \leq -\frac{a_1A}{2}d(z). \nonumber
    \end{aligned}
    \end{equation}
    since  
     \begin{equation}
    \begin{aligned}
    \sum_{\alpha=1}^{n-1}\mu_\alpha T_{\alpha}^i \bar T_{\alpha}^j d_{i\bar j}
 =\,&   \sum_{\alpha=1}^{n-1} \mu_\alpha\sigma_{\alpha\bar\alpha}(z)+
  \tau\sum_{\alpha=1}^{n-1}\mu_\alpha
  \\\,&+\sum_{\alpha=1}^{n-1}\mu_\alpha \left(T_{\alpha}^i \bar T_{\alpha}^j (z)- T_{\alpha}^i \bar T_{\alpha}^j (0)\right)d_{i\bar j} \\
 =\,& -a_1+\tau+O(|z|) \leq -\frac{a_1}{4}  \nonumber
    \end{aligned}
    \end{equation}
    provided one chooses $0<\delta, \tau\ll1.$
Here we also use \eqref{key-1-yuan3},
    \begin{equation}
    \label{yuan3-buchong2}
    \begin{aligned}
    \sum_{\alpha=1}^{n-1} \mu_\alpha T_{\alpha}^i \bar T_{\alpha}^j(z)
    =\sum_{\alpha=1}^{n-1} \mu_\alpha T_{\alpha}^i \bar T_{\alpha}^j(0)+O(|z|)
    =\sum_{\alpha=1}^{n-1} \delta_{\alpha i} \delta_{\alpha j}+O(|z|),
       \end{aligned}
    \end{equation}
    and
\begin{equation}
     \label{c3-yuan3}
     \begin{aligned}
     \sum_{\alpha=1}^{n-1}\mu_\alpha\sigma_{\alpha\bar\alpha}(z)=\sum_{\alpha=1}^{n-1}\mu_\alpha\sigma_{\alpha\bar\alpha}(0)+O(|z|).
     \end{aligned}
     \end{equation}  
\item  
 \begin{equation}
    \begin{aligned}
  \,& 
     l(z)\sum_{\alpha=1}^{n-1}\mu_\alpha T^i_\alpha \bar T^j_\alpha
    \sigma_{i\bar j}+\sum_{\alpha=1}^{n-1}\mu_{\alpha}T^i_\alpha \bar T^j_\alpha(l_i\sigma_{\bar j}
    +\sigma_i l_{\bar j}) \\
   = \,&
    l(z) \sum_{\alpha=1}^{n-1}\mu_\alpha\sigma_{\alpha\bar\alpha}(0)
    - \tau\sum_{\alpha=1}^{n-1} \mu_\alpha (z_\alpha l_\alpha+\bar z_{\alpha}  \bar l_{\alpha})
    +O(|z|^2) \nonumber
    \end{aligned}
    \end{equation}
    since by \eqref{yuan3-buchong2},  
    $ \sum_{i=1}^n T_\alpha^i \sigma_i=-\tau\sum_{i=1}^n T_\alpha^i
\bar z_i $
we have
    \begin{equation}
    \begin{aligned}
     l(z) \sum_{\alpha=1}^{n-1}\mu_\alpha T^i_\alpha \bar T^j_\alpha \sigma_{i\bar j} 
    = l(z) \sum_{\alpha=1}^{n-1}\mu_\alpha\sigma_{\alpha\bar\alpha}(0)
+O(|z|^2),   \nonumber
    \end{aligned}
    \end{equation}
    \begin{equation}
    \begin{aligned}
       \sum_{\alpha=1}^{n-1}\mu_\alpha T^i_\alpha \bar T^j_\alpha(l_i \sigma_{\bar j} + \sigma_i l_{\bar j})
   =-\tau \sum_{\alpha=1}^{n-1} \mu_\alpha(\bar l_{\alpha} \bar z_{\alpha}+l_\alpha z_\alpha)
+O(|z|^2).   \nonumber
    \end{aligned}
    \end{equation}
    \end{itemize}
    Putting these together, 
     \begin{equation}
     \label{together1}
    \begin{aligned}
    \Lambda_\mu(\mathfrak{g}[w])\leq\,&
  \sum_{\alpha=1}^{n-1} 2\mathfrak{Re}
  \left\{z_\alpha(k_\alpha 
  -\tau\mu_\alpha l_\alpha
  +l_\alpha\sum_{\beta=1}^{n-1}\mu_\beta \sigma_{\beta\bar\beta}(0))\right\}
    \\ \,& + 2\mathfrak{Re}\left\{z_n(k_n+l_n\sum_{\beta=1}^{n-1}\mu_\beta \sigma_{\beta\bar\beta}(0))\right\} 
    -\frac{Aa_1}{2}d(z) + O(|z|^2).   \nonumber
    \end{aligned}
    \end{equation}
   Let $l_n=-\frac{k_n}{\sum_{\beta=1}^{n-1} \mu_\beta \sigma_{\beta\bar\beta}(0)}$.
  For $1\leq \alpha\leq n-1$, we set
  \begin{equation}
  \label{chosen-1}
    \begin{aligned}
    l_\alpha=-\frac{k_\alpha}{\sum_{\beta=1}^{n-1}\mu_\beta \sigma_{\beta\bar\beta}(0)-\tau \mu_\alpha}.
    \end{aligned}
    \end{equation}
From $\mu_\alpha\geq 0$ and \eqref{key-1-yuan3}, we see
    such $l_i$ (or equivalently the  $l(z)$) are all well defined and are uniformly bounded.
    
    We thus complete the proof if $0<\tau, \delta\ll1$, $A\gg1$.
    \end{proof}

    \subsection{Completion of the proof of Lemma \ref{keylemma1-yuan3}}
  Let $w$ be as in Lemma \ref{lemma-key2-yuan3}. From the construction above, we know that there is a uniform positive constant $C_1'$ such that
  $$ |\mathfrak{g}[w]|_{C^0(\Omega_\sigma)}\leq C_1'.$$
    Let $\lambda[w]=\lambda_{\omega}(\mathfrak{g}[w])$. Assume 
    $\lambda_1[w]\leq \cdots\leq \lambda_n[w]$.
    Lemma \ref{lemma-key2-yuan3}, together with Lemma \ref{lemma-yuan3-buchong1}, implies
    \begin{equation}
    \begin{aligned}
    \sum_{\alpha=1}^{n-1} \mu_\alpha \lambda_\alpha[w]\leq 0 \mbox{ in } \Omega_\delta.  \nonumber
    \end{aligned}
    \end{equation}
So $(\lambda_1[w],\cdots,\lambda_{n-1}[w])\notin\Gamma_\infty$ by \eqref{key-18-yuan3}. In other words, $\lambda[w]\in X$, where
\[X=\{\lambda\in\mathbb{R}^n: \lambda'\in \mathbb{R}^{n-1}\setminus \Gamma_\infty\}\cap \{\lambda\in\mathbb{R}^n:   |\lambda|\leq C_1'\}.\]
Let $$\bar\Gamma^{\inf_M\psi}=\{\lambda\in\Gamma: f(\lambda)\geq \inf_M\psi\}$$
 Notice that $\Gamma_\infty$ is open so $X$ is a compact subset; furthermore $X\cap \bar\Gamma^{\inf_M\psi}=\emptyset$. 
So we can deduce that the distance between $\bar\Gamma^{\inf_M\psi}$ and $X$ 
is greater than some positive constant depending on  
$\delta_{\psi,f}$ and other known data.
Therefore, there exists an $\epsilon>0$  such that for any $z\in\Omega_\delta$
\begin{equation}
    \begin{aligned}
    \epsilon\vec{\bf 1} +\lambda[w]\notin \bar\Gamma^{\inf_M\psi}.  \nonumber
    \end{aligned}
    \end{equation}

Since one can choose a  positive constant $C'$ such that $ x_n\leq C'|z|^2$ on 
$\partial M\cap \bar{\Omega}_\delta$, 
there is a positive constant $C_2$ depending only on $M$ and $\delta$ so that  
$$x_n\leq C_2 |z|^2 \mbox{ on }\partial\Omega_\sigma.$$

 Let $\epsilon$ and $C_2$ be as above, we define ${h}(z)=w(z)+\epsilon (|z|^2-\frac{x_n}{C_2})$. Thus
    \begin{equation}
    \begin{aligned}
    u\leq {h} \mbox{ on } \partial \Omega_\delta.  \nonumber
    \end{aligned}
    \end{equation}
    Moreover, $\chi_{i\bar j}+{h}_{i\bar j}=(\chi_{i\bar j}+w_{i\bar j})+\epsilon\delta_{ij}$ so $\lambda[{h}]\notin \bar\Gamma^{\inf_M\psi}.$
    By \cite[Lemma B]{CNS3}, we have $$u\leq {h} \mbox{ in }\Omega_\delta.$$
   Notice 
   $u(0)=\varphi(0)$ and ${h}(0)=\varphi(0)$, we have $u_{x_n}(0)\leq h_{x_n}(0)$ then
    \begin{equation}
    \begin{aligned}
    t_0\leq \frac{1}{1+\epsilon/(\eta C_2)}, \mbox{ i.e., }  (1-t_0)^{-1}\leq 1+\frac{\eta C_2}{\epsilon}.  \nonumber
    \end{aligned}
    \end{equation}

 \begin{remark}
   In fact, \eqref{key-14-yuan3} and \eqref{yuan3-buchong5} imply $|u|_{C^0(M)}+|\nabla u|_{C^0(\partial M)}\leq C$.
    \end{remark}
    
\begin{remark}
The discussions above work for more general equations of the form
\begin{equation}
\begin{aligned}
f(\lambda(\chi+\sqrt{-1}\partial\overline{\partial}u+\sqrt{-1}\partial u\wedge \overline{\zeta}+\sqrt{-1}\zeta\wedge\overline{\partial}u))=\psi \nonumber
\end{aligned}
\end{equation}
where $\zeta=\sum_{i=1}^n\zeta_i dz^i$ is a smooth $(1,0)$-form.
\end{remark}

     \section{Proof of Proposition \ref{proposition-quar-yuan2}}
   \label{sec4}

  The equation \eqref{MA-n-1} can be reduced to 
\begin{equation}
\label{mainequ-gauduchon}
\begin{aligned}
\log P_{n-1}(\lambda(\tilde{\mathfrak{g}}[u]))=\psi \mbox{ in }M,
\end{aligned}
\end{equation}
for $\psi=(n-1)\phi+n\log (n-1)$, where  
 \begin{equation}
 \label{yuan3-buchong111}
 \begin{aligned}
 \tilde{\mathfrak{g}}_{i\bar j}=u_{i\bar j}+\tilde{\chi}_{i\bar j}+W_{i\bar j},
 \end{aligned}
 \end{equation}
here $\check{\chi}_{i\bar j}=(\mathrm{tr}_\omega \tilde{\chi})g_{i\bar j}-(n-1)\tilde{\chi}_{i\bar j}$,
$W_{i\bar j}=(\mathrm{tr}_\omega Z)g_{i\bar j}-(n-1)Z_{i\bar j}$. 
Locally,
\begin{equation}
\label{Z-tensor1}
\begin{aligned}
Z_{i\bar j}=\,& \frac{g^{p\bar q}   \bar T^l_{ql} g_{i\bar j}u_p
+g^{p\bar q} T^{k}_{pk} g_{i\bar j}u_{\bar q}
-g^{k\bar l}g_{i\bar q} \bar T^q_{lj}u_k
 -g^{k\bar l}g_{q\bar j}T^{q}_{ki}u_{\bar l}
-\bar T^{l}_{jl}u_i -T^{k}_{ik}u_{\bar j} }{2(n-1)}, 
\end{aligned}
\end{equation}
where $T^k_{ij}$ are the torsions,
see also \cite{STW17}.
Here 
\begin{equation}
\begin{aligned}
\,& f(\lambda)=\log P_{n-1}(\lambda) = \sum_{i=1}^n \log\mu_i,  \,& \mu_i=\sum_{j\neq i}\lambda_j  \nonumber
\end{aligned}
\end{equation}
with corresponding cone 
  \begin{equation}
 \label{yuan3-buchong112}
 \begin{aligned}
 \mathcal{P}_{n-1}=\{\lambda\in \mathbb{R}^n: \mu_i>0, \mbox{  } \forall 1\leq i\leq n\}.
  \end{aligned}
 \end{equation}
One can verify that condition \eqref{gamma-cone} is equivalent to 
$\lambda(\tilde{\mathfrak{g}}[u])\in \mathcal{P}_{n-1} \mbox{ in } \bar M,$
which allows one to seek the solutions of  \eqref{mainequ-gauduchon} or equivalently  \eqref{MA-n-1}
 within the framework of elliptic equations, since
 $f(\lambda)=\log P_{n-1}(\lambda)$ satisfies \eqref{elliptic}, \eqref{concave} and  \eqref{addistruc} in  $\mathcal{P}_{n-1}$.

\subsection{Preliminaries}
  Throughout this section we denote 
   \[\mathfrak{\tilde{g}}=\mathfrak{\tilde{g}}[u],  \mbox{  } Z=Z[u], \mbox{  } W=W[u], \mbox{  } \mathfrak{\underline{\tilde{g}}}=\mathfrak{\tilde{g}}[\underline{u}],  \mbox{ } \underline{Z}=Z[\underline{u}], \mbox{ } \underline{W}=W[\underline{u}]\]
for solution $u$ and subsolution $\underline{u}$.

Fix $x_0\in \partial M$. 
Around $x_0$ we set local holomorphic coordinates 
$(z_1,\cdots,z_n)$, $z_i=x_i+\sqrt{-1}y_i$, centered at $x_0$, such that $g_{i\bar j}(0)=\delta_{ij}$ and $\frac{\partial}{\partial x_n}$ is the inner normal vector at the origin, and ${T_{x_0}}_{\partial M}^{1,0}$ is spanned by $\frac{\partial}{\partial z_\alpha}$
for $1\leq\alpha\leq n-1$.

 By \eqref{Z-tensor1}, \eqref{yuan3-buchong5} and $W[v]_{i\bar j}=(\mathrm{tr}_\omega Z[v])g_{i\bar j}-(n-1)Z[v]_{i\bar j}$, one can verify that at the origin ($\{z=0\}$)  
 \begin{equation}
 \label{yuan3-buchong11}
    \begin{aligned}
    \sum_{\alpha=1}^{n-1} W[v]_{\alpha\bar\alpha}= (n-1)Z[v]_{n\bar n},  
    \end{aligned}
    \end{equation}
\begin{equation}
 \label{yuan3-buchong12}
\begin{aligned} 
\sum_{\alpha=1}^{n-1}\tilde{\mathfrak{g}}[v]_{\alpha\bar\alpha} 
=  \sum_{\alpha=1}^{n-1}(v_{\alpha\bar\alpha}+ \check{\chi}_{\alpha\bar\alpha})+ \sum_{\alpha=1}^{n-1}W[v]_{\alpha\bar\alpha},
\end{aligned}
\end{equation}
\begin{equation}
\label{stw-44}
\begin{aligned}
2(n-1)Z[v]_{n\bar n}
=\,&  \sum_{\alpha, \beta=1}^{n-1}  (\bar T^\beta_{\alpha \beta}  v_\alpha
+ T^{\beta}_{\alpha \beta}  v_{\bar \alpha}), 
\end{aligned}
\end{equation}

\begin{equation}
\label{lemma-B5}
\begin{aligned} \sum_{\alpha=1}^{n-1}\tilde{\mathfrak{g}}_{\alpha\bar\alpha} =
  \sum_{\alpha=1}^{n-1}\underline{\tilde{\mathfrak{g}}}_{\alpha\bar\alpha}
  +(u-\underline{u})_{x_n}\sum_{\alpha=1}^{n-1}\sigma_{\alpha\bar\alpha}. 
\end{aligned}
\end{equation}

We rewrite $W[v]_{i\bar j}$ as follows
\begin{equation}
\begin{aligned}
W[v]_{i\bar j} = W_{i\bar j}^k v_k + W_{i\bar j}^{\bar k} v_{\bar k}. \nonumber
\end{aligned}
\end{equation}
One can see   
 \begin{equation}
 \label{lemma-B4}
    \begin{aligned}
    \sum_{\alpha=1}^{n-1} W_{\alpha\bar\alpha}^n = \sum_{\alpha=1}^{n-1} W_{\alpha\bar\alpha}^{\bar n}=0
    \end{aligned}
    \end{equation}
 at the origin.



\subsection{A key ingredient and its proof}

As in Section \ref{sec3}, we set $\eta=(u-\underline{u})_{x_n}(0)$. 
We know that $\eta\geq 0$. 
Let  
\begin{equation}
\label{t0-2}
    \begin{aligned}
    t_0=-{\eta \sum_{\alpha=1}^{n-1}\sigma_{\alpha\bar\alpha}(0)}/{\sum_{\alpha=1}^{n-1}\underline{\tilde{\mathfrak{g}}}_{\alpha\bar\alpha}(0)}.
    \end{aligned}
    \end{equation}
     We assume throughout $\eta>0$. (Otherwise $t_0=0$ and the proof is almost parallel to that given in \cite{yuan2019}).
From $\lambda(\underline{\tilde{\mathfrak{g}}})\in \mathcal{P}_{n-1}$ we know $\sum_{\alpha=1}^{n-1}\underline{\tilde{\mathfrak{g}}}_{\alpha\bar\alpha}(0)>0$.
    Clearly $t_0<1$ since $\sum_{\alpha=1}^{n-1}{\tilde{\mathfrak{g}}}_{\alpha\bar\alpha}(0)>0$ ($\lambda({\tilde{\mathfrak{g}}})\in \mathcal{P}_{n-1}$).
    \begin{lemma}
    \label{key-lemma-B1}
    There is a uniform positive constant $C$ depending on $(\delta_{\psi,f})^{-1}$, $|u|_{C^0(\bar M)}$,  $|\nabla u|_{C^0(\partial M)}$, $|\underline{u}|_{C^2(\bar M)}$, 
    $\partial M$ up to third derivatives and other known data such that
    \[(1-t_0)^{-1}\leq C.\]
    \end{lemma}
  In what follows we assume $t_0>\frac{1}{2}.$ Since $\eta$ has a uniform upper bound, thus at origin
  \begin{equation}
  \label{yuan3-a2}
  -\sum_{\alpha=1}^{n-1}\sigma_{\alpha\bar\alpha}(0)\geq {t_0\sum_{\alpha=1}^{n-1}\underline{\tilde{\mathfrak{g}}}_{\alpha\bar\alpha}(0)}/{\eta}\geq  {\sum_{\alpha=1}^{n-1}\underline{\tilde{\mathfrak{g}}}_{\alpha\bar\alpha}(0)}/{2\eta}\geq a_2
  \end{equation}
  where $$a_2= {\inf_{z\in\partial M}\sum_{\alpha=1}^{n-1}\underline{\tilde{\mathfrak{g}}}_{\alpha\bar\alpha}(z)}/{2\sup_{\partial M}|\nabla (u-\underline{u})|}.$$
 As in Section \ref{sec3} we set on $\Omega_\delta$
\begin{equation}
    \begin{aligned}
    d(z)=\sigma(z)+\tau |z|^2 \nonumber
    \end{aligned}
    \end{equation}
    where $\tau$ is a positive constant 
     to be determined; and let 
    \begin{equation}
    \begin{aligned}
    w(z)=\underline{u}(z)+({\eta}/{t_0})\sigma(z)+l(z)\sigma(z)+Ad(z)^2. \nonumber
    \end{aligned}
    \end{equation}
    where $A$ is a positive constant to be determined, and $l(z)=\sum_{i=1}^n(l_iz_i+\bar l_{ i} \bar z_{i})$ 
    where $l_i\in \mathbb{C}$, $\bar l_i=l_{\bar i}$ to be chosen as in \eqref{chosen-2}.

 As in Section \ref{sec3}, let $T_\alpha=\sum_{k=1}^nT_\alpha^k\frac{\partial }{\partial z_k}$  be an orthonormal basis of holomorphic tangent space of level 
 hypersurface $\{w: d(w)=d(z)\}$ at $z$, $1\leq\alpha\leq n-1$, such that at origin $T_\alpha(0)=\frac{\partial}{\partial z_\alpha}$.
 Let's define a local operator $\Lambda$: For a real $(1,1)$-form $\Theta=\sqrt{-1}\Theta_{i\bar j}dz_i\wedge d\bar z_j,$
     \begin{equation}
    \begin{aligned}
    \Lambda(\Theta)=\sum_{\alpha=1}^{n-1} T_{\alpha}^i \bar T_{\alpha}^j \Theta_{i\bar j}.  \nonumber
    \end{aligned}
    \end{equation}
    
    \begin{lemma}
   \label{key-lemma-B2}
   There are parameters $\tau$, $\delta$, $A$ and 
   $l(z)$ depending on 
   $|u|_{C^0(M)}$, 
   $|\nabla u|_{C^0(\partial M)}$,  
   $|\underline{u}|_{C^2(M)}$, 
   $\partial M$ up to third derivatives and other known data such that
    \begin{equation}
    \begin{aligned}
\,&   \Lambda(\tilde{\mathfrak{g}}[w]) \leq0 \mbox{ in } \Omega_\delta, \,& u\leq w \mbox{ on } \partial \Omega_\delta. \nonumber
    \end{aligned}
    \end{equation}
    
    \end{lemma}

   \begin{proof}
  Direct computations give
    \begin{equation}
    \begin{aligned}
    w_i = \underline{u}_i + \frac{\eta}{t_0} \sigma_i +l_i\sigma +l(z)\sigma_i +2Add_i \nonumber
    \end{aligned}
    \end{equation}
   \begin{equation}
    \begin{aligned}
    w_{i\bar j}=
    \underline{u}_{i\bar j}+\frac{\eta}{t_0} \sigma_{i\bar j}+l(z)\sigma_{i\bar j}
    +(l_i\sigma_{\bar j}+\sigma_i l_{\bar j})
    +2Add_{i\bar j} +2A d_i d_{\bar j}.  \nonumber
    \end{aligned}
    \end{equation}
   Then
     \begin{equation}
    \begin{aligned}
    \Lambda(\tilde{\mathfrak{g}}[w])=\,& 
     \sum_{\alpha=1}^{n-1}T_\alpha^{i}\bar T_\alpha^j ( (\check{\chi}_{i\bar j}+ \underline{w}_{i\bar j}+W_{i\bar j}^pw_p + W_{i\bar j}^{\bar q}w_{\bar q}) \\
    =\,&   \sum_{\alpha=1}^{n-1}T_\alpha^{i}\bar T_\alpha^j  
    (\check{\chi}_{i\bar j}+ \underline{u}_{i\bar j}+W_{i\bar j}^p \underline{u}_p+W_{i\bar j}^{\bar q}\underline{u}_{\bar q}+\frac{\eta}{t_0} \sigma_{i\bar j})
    \\\,&
    +l(z) \sum_{\alpha=1}^{n-1}T_\alpha^{i}\bar T_\alpha^j \sigma_{i\bar j}
    +\sum_{\alpha=1}^{n-1}T_\alpha^{i}\bar T_\alpha^j (\sigma_i l_{\bar j}+l_i\sigma_{\bar j})
    \\ \,&
    +2Ad(z)\sum_{\alpha=1}^{n-1}T_\alpha^{i}\bar T_\alpha^j d_{i\bar j}
    +\frac{\eta}{t_0} \sum_{\alpha=1}^{n-1}T_\alpha^{i}\bar T_\alpha^j (W_{i\bar j}^p \sigma_p+ W_{i\bar j}^{\bar q} \sigma_{\bar q})\\\,&
      +l(z)\sum_{\alpha=1}^{n-1}T_\alpha^{i}\bar T_\alpha^j (W_{i\bar j}^p\sigma_p+W_{i\bar j}^{\bar q}\sigma_{\bar q})
     +\sum_{\alpha=1}^{n-1}T_\alpha^{i}\bar T_\alpha^j (W_{i\bar j}^pl_p+W_{i\bar j}^{\bar q}l_{\bar q})\sigma
          \\\,&
          +2Ad(z)\sum_{\alpha=1}^{n-1}T_\alpha^{i}\bar T_\alpha^j(W_{i\bar j}^pd_p+W_{i\bar j}^{\bar q}d_{\bar q}).  \nonumber
    \end{aligned}
    \end{equation}
    \begin{itemize}
 \item   At origin $z=0$, $T_{\alpha}^i=\delta_{\alpha i}$, so 
      \begin{equation}
    \begin{aligned}
    \,& \sum_{\alpha=1}^{n-1}T_\alpha^{i}\bar T_\alpha^j  
    (\check{\chi}_{i\bar j}+ \underline{u}_{i\bar j}+W_{i\bar j}^p \underline{u}_p+W_{i\bar j}^{\bar q}\underline{u}_{\bar q}+\frac{\eta}{t_0} \sigma_{i\bar j})(0)
    \\ =\,&\sum_{\alpha=1}^{n-1} \underline{\tilde{\mathfrak{g}}}_{\alpha\bar\alpha}(0)+\frac{\eta}{t_0}\sum_{\alpha=1}^{n-1}\sigma_{\alpha\bar \alpha}(0)=0.  \nonumber
    \end{aligned}
    \end{equation}
     Thus there are complex constants $k_i$ such that on $\Omega_\sigma$,
       \begin{equation}
    \begin{aligned}
   \,&  \sum_{\alpha=1}^{n-1}T_\alpha^{i}\bar T_\alpha^j  
    (\check{\chi}_{i\bar j}+ \underline{u}_{i\bar j}+W_{i\bar j}^p \underline{u}_p+W_{i\bar j}^{\bar q}\underline{u}_{\bar q}+\frac{\eta}{t_0} \sigma_{i\bar j}) \\
     =\,&\sum_{i=1}^n(k_iz_i+\bar k_i\bar z_i)+O(|z|^2).  \nonumber
    \end{aligned}
    \end{equation}
  \item  Next, we see
        \begin{equation}
    \begin{aligned}
    2A d(z)  \sum_{\alpha=1}^{n-1}T_\alpha^i\bar T_\alpha^{j} d_{i\bar j} \leq -\frac{Aa_2d(z)}{2},  \nonumber
    \end{aligned}
    \end{equation}
   provided $0<\delta, \tau\ll1$, since 
  \begin{equation}
    \begin{aligned}
    \sum_{\alpha=1}^{n-1}T_\alpha^i\bar T_\alpha^{j} d_{i\bar j}
    =\,& (\sum_{\alpha=1}^{n-1}T_\alpha^i\bar T_\alpha^{j}-\sum_{\alpha=1}^{n-1}T_\alpha^i\bar T_\alpha^{j} (0) )d_{i\bar j}
    +\sum_{\alpha=1}^{n-1}\sigma_{\alpha\bar\alpha}(z)+(n-1)\tau \\
   =\,& (n-1)\tau-a_2+O(|z|) \leq -\frac{a_2}{4} \nonumber
    \end{aligned}
    \end{equation}
    by  \eqref{yuan3-buchong2} and $\sum_{\alpha=1}^{n-1}\sigma_{\alpha\bar\alpha}(z)=\sum_{\alpha=1}^{n-1}\sigma_{\alpha\bar\alpha}(0)+O(|z|)$.
    
   \item   
  \begin{equation}
    \begin{aligned}
  \,& l(z) \sum_{\alpha=1}^{n-1}T_\alpha^{i}\bar T_\alpha^j \sigma_{i\bar j}
    +\sum_{\alpha=1}^{n-1}T_\alpha^{i}\bar T_\alpha^j (\sigma_i l_{\bar j}+l_i\sigma_{\bar j}) \\
    = \,&
    l(z)\sum_{\alpha=1}^{n-1}\sigma_{\alpha\bar\alpha}(0)
    + \tau\sum_{\alpha=1}^{n-1}  (z_\alpha l_\alpha+\bar z_{\alpha} \bar l_{\alpha})
    +O(|z|^2)  \nonumber
    \end{aligned}
    \end{equation}
    since by \eqref{yuan3-buchong2} and 
    $ \sum_{i=1}^n T_\alpha^i \sigma_i=-\tau\sum_{i=1}^n T_\alpha^i
\bar z_i $
one has
    \begin{equation}
    \begin{aligned}
      l(z) \sum_{\alpha=1}^{n-1}  T^i_\alpha \bar T^j_\alpha \sigma_{i\bar j} 
    =
    l(z) \sum_{\alpha=1}^{n-1}\mu_\alpha\sigma_{\alpha\bar\alpha}(0)
+O(|z|^2)  \nonumber
    \end{aligned}
    \end{equation}
    \[ \sum_{\alpha=1}^{n-1}T_\alpha^{i}\bar T_\alpha^j (\sigma_i l_{\bar j}+l_i\sigma_{\bar j})=-\tau\sum_{\alpha=1}^{n-1}  (z_\alpha l_\alpha+\bar z_{\alpha}  \bar l_{\alpha})
    +O(|z|^2).\]

\item 
At the origin, 
 \begin{equation}
    \begin{aligned}
   \,&  \sum_{\alpha=1}^{n-1}T_\alpha^{i}\bar T_\alpha^j (W_{i\bar j}^p \sigma_p+ W_{i\bar j}^{\bar q} \sigma_{\bar q}) (0)\\
   =\,&  \sum_{\alpha,\beta=1}^{n-1}  (W_{\alpha\bar \alpha}^\beta \sigma_\beta+ W_{\alpha\bar \alpha}^{\bar \beta} \sigma_{\bar \beta}) (0)  + \sum_{\alpha=1}^{n-1} (W_{\alpha\bar \alpha}^n \sigma_n+ W_{\alpha\bar \alpha}^{\bar n} \sigma_{\bar n})(0)=0, \nonumber
    \end{aligned}
    \end{equation}
   since $\sigma_\beta(0)=0$, and by  \eqref{lemma-B4}
     \[\sum_{\alpha=1}^{n-1}W_{\alpha\bar\alpha}^n(0)=0, \mbox{  }\sum_{\alpha=1}^{n-1}W_{\alpha\bar\alpha}^{\bar n}(0)=0.\]
    Thus on $\Omega_\sigma$, 
      \begin{equation}
    \begin{aligned}
    l(z)\sum_{\alpha=1}^{n-1}T_\alpha^{i}\bar T_\alpha^j (W_{i\bar j}^p\sigma_p+W_{i\bar j}^{\bar q}\sigma_{\bar q})(z)
    =O(|z|^2), \nonumber
    \end{aligned}
    \end{equation}
    and there are complex constants $m_i$ such that 
        \begin{equation}
    \begin{aligned}
    \frac{\eta}{t_0} \sum_{\alpha=1}^{n-1}T_\alpha^{i}\bar T_\alpha^j (W_{i\bar j}^p \sigma_p+ W_{i\bar j}^{\bar q} \sigma_{\bar q})(z)
    =\sum_{i=1}^n(m_iz_i+\bar m_i \bar z_i)+O(|z|^2).  \nonumber
    \end{aligned}
    \end{equation}
    \item Similarly 
    $\sum_{\alpha=1}^{n-1}T_\alpha^{i}\bar T_\alpha^j(W_{i\bar j}^pd_p+W_{i\bar j}^{\bar q}d_{\bar q})(0)=0$, thus on $\Omega_\delta$,
     $$\sum_{\alpha=1}^{n-1}T_\alpha^{i}\bar T_\alpha^j(W_{i\bar j}^pd_p+W_{i\bar j}^{\bar q}d_{\bar q})(z)=O(|z|)$$ and so
 \begin{equation}
    \begin{aligned}
    2Ad(z)\sum_{\alpha=1}^{n-1}T_\alpha^{i}\bar T_\alpha^j(W_{i\bar j}^pd_p+W_{i\bar j}^{\bar q}d_{\bar q})(z)=Ad(z)O(|z|). \nonumber
    \end{aligned}
    \end{equation}

    \item Finally
      \begin{equation}
    \begin{aligned}
    \sum_{\alpha=1}^{n-1}T_\alpha^{i}\bar T_\alpha^j (W_{i\bar j}^pl_p+W_{i\bar j}^{\bar q}l_{\bar q})\sigma(z) \leq C_1\sigma(z). \nonumber
    \end{aligned}
    \end{equation}
    
      \end{itemize}
    Therefore, we get
   \begin{equation}
    \begin{aligned}
    \Lambda(\tilde{\mathfrak{g}}[w])\leq \,& 
    2\mathfrak{Re}\sum_{\alpha=1}^{n-1}\left[z_\alpha\left(k_\alpha+m_\alpha+ l_\alpha\left(\sum_{\beta=1}^{n-1}\sigma_{\beta\bar\beta}(0)-\tau\right)\right)\right] \\
    \,&
     +2\mathfrak{Re} \left[z_n\left(k_n+m_n + l_n\sum_{\beta=1}^{n-1}\sigma_{\beta\bar\beta}(0))\right)\right] \\
     \,&
     -\frac{a_2 A d(z)}{2} +Ad(z)O(|z|)+C_1\sigma(z) + O(|z|^2).  \nonumber
    \end{aligned}
    \end{equation}
    We complete the proof if $0<\tau, \delta\ll1$, $A\gg1$, and we set 
     \begin{equation}
     \label{chosen-2}
    \begin{aligned}
    l_\alpha=-\frac{k_\alpha+m_\alpha}{\sum_{\beta=1}^{n-1}\sigma_{\beta\bar\beta}(0)-\tau} \mbox{ for } 1\leq \alpha\leq n-1, \mbox{ }   l_n=-\frac{k_n+m_n}{\sum_{\beta=1}^{n-1}\sigma_{\beta\bar\beta}(0)}.
    \end{aligned}
    \end{equation}
   We can see each $|l_i|$ is uniformly bounded, since $\sum_{\beta=1}^{n-1}\sigma_{\beta\bar\beta}(0)\leq -a_2<0$. 
    \end{proof}

\subsection{Proof of Lemma \ref{key-lemma-B1}}

Let $\lambda(\tilde{\mathfrak{g}}[w])=(\lambda_1[w],\cdots,\lambda_n[w])$, let  $\mu_i[w]=\sum_{j\neq i}\lambda_j[w]$ and we assume $\lambda_1[w]\leq \cdots\leq \lambda_n[w]$. 
Denote by
    \begin{equation}
    \begin{aligned}
    \mathcal{P}_{n-1}^{\inf_M\psi}=\left\{\lambda\in\mathcal{P}_{n-1}: \sum_{i=1}^n \log\mu_i\geq \inf_M \psi \right\}. \nonumber
    \end{aligned}
    \end{equation}
     
    As in section \ref{sec3} we set ${h}(z)=w(z)+\epsilon (|z|^2-\frac{x_n}{C_2})$ where $C_2$ be chosen so that $x_n\leq C_2 |z|^2$ on $\partial\Omega_\delta.$
    Thus
    \begin{equation}
    \begin{aligned}
    u\leq {h} \mbox{ on } \partial \Omega_\delta. \nonumber
    \end{aligned}
    \end{equation}
    Lemmas \ref{key-lemma-B2} and \ref{lemma-yuan3-buchong1} give
 \begin{equation}
    \begin{aligned}
    \sum_{\alpha=1}^{n-1}\lambda_\alpha[w]\leq 0 \mbox{ in } \Omega_\delta.  \nonumber
    \end{aligned}
    \end{equation}
    That is, in $\Omega_\delta$,
     \begin{equation}
    \begin{aligned}
    \lambda[w]\notin \mathcal{P}_{n-1}, \mbox{ i.e. } \mu[w]\notin\Gamma_n.  \nonumber
    \end{aligned}
    \end{equation}
    Then there is $0<\epsilon\ll1$ depending on $\delta_{\psi,f}$, $\lambda[w]$, torsion tensor and other known data such that
\begin{equation}
    \begin{aligned}
   \lambda[{h}] \notin \mathcal{P}_{n-1}^{\inf_M\psi}. \nonumber
    \end{aligned}
    \end{equation}
      By \cite[Lemma B]{CNS3} again, we have $$u\leq {h} \mbox{ in }\Omega_\delta.$$
   Notice 
   $u(0)=\varphi(0)$ and ${h}(0)=\varphi(0)$, we have  $(u-{h})_{x_n}(0)\leq 0$ then   
    \begin{equation}
    \begin{aligned}
    (1-t_0)^{-1}\leq 1+\frac{\eta C_2}{\epsilon}.  \nonumber
    \end{aligned}
    \end{equation}

\subsection{Completion of proof of Proposition \ref{proposition-quar-yuan2}}

Around $x_0$ we use the local holomorphic coordinates  
 we have chosen above;
furthermore, we assume that $ ({\tilde{\mathfrak{g}}}_{\alpha\bar\beta})$ is diagonal at the origin ($x_0=\{z=0\}$).
 In the proof the discussion is done at the origin, and 
 the Greek letters, such as $\alpha, \beta$, range from $1$ to $n-1$.
Let  $\mu_i=\sum_{j\neq i} \lambda_j$, and 
 $$\tilde{f}(\mu)=f(\lambda)=\sum_{i=1}^n \log\mu_i.$$
Let's denote
\begin{equation}
{\tilde{A}}(R)=\left(
\begin{matrix}
R-\tilde{\mathfrak{{g}}}_{1\bar 1}&&  &-\tilde{\mathfrak{g}}_{1 \bar n}\\
&\ddots&&\vdots \\
& &  R-\tilde{\mathfrak{{g}}}_{{(n-1)} \overline{(n-1)}}&- \tilde{\mathfrak{g}}_{(n-1) \bar n}\\
-\tilde{\mathfrak{g}}_{n \bar 1}&\cdots& -\tilde{\mathfrak{g}}_{n \overline{(n-1)}}& \sum_{\alpha=1}^{n-1}\tilde{\mathfrak{g}}_{\alpha \bar\alpha}  \nonumber
\end{matrix}
\right).
\end{equation}
In particular, when $R=\mathrm{tr}_\omega (\tilde{\mathfrak{{g}}})$, ${\tilde{A}}(R)=\mathrm{tr}_\omega(\mathfrak{\tilde{g}}) \omega-\mathfrak{\tilde{g}}$.
By \eqref{lemma-B5} and \eqref{t0-2}, $\tilde{{A}}(R)$ can be rewritten as
\begin{equation}
\tilde{\underline{A}}(R)=\left(
\begin{matrix}
R-\tilde{\mathfrak{{g}}}_{1\bar 1}&&  &-\tilde{\mathfrak{g}}_{1\bar n}\\
&\ddots&&\vdots \\
& & R-\tilde{\mathfrak{{g}}}_{{(n-1)}  \overline{(n-1)}}&- \tilde{\mathfrak{g}}_{(n-1) \bar n}\\
-\tilde{\mathfrak{g}}_{n \bar1}&\cdots& -\tilde{\mathfrak{g}}_{n \overline{(n-1)}}& (1-t_0)\sum_{\alpha=1}^{n-1}\underline{\tilde{\mathfrak{g}}}_{\alpha\bar\alpha} \nonumber
\end{matrix}
\right).
\end{equation}
That is $\tilde{\underline{A}}(R)={\tilde{A}}(R).$
Similar as before, there is a uniform positive constant $R_0$ depending on $(1-t_0)^{-1}$ and $(\inf_{\partial M}\mathrm{dist}(\lambda(\mathfrak{\tilde{g}}),\partial \Gamma_n))^{-1}$ (but not on $(\delta_{\psi,f})^{-1}$) such that
$$\tilde{f}\left(R_0,\cdots, R_0,(1-t_0)\sum_{\alpha=1}^{n-1}\underline{\tilde{\mathfrak{g}}}_{\alpha \bar\alpha}\right) > 
\psi.$$
Therefore, there is a positive constant $\varepsilon_{0}$, 
 depending  on $\inf_{\partial M}\mathrm{dist}(\lambda(\mathfrak{\tilde{g}}),\partial \Gamma_n)$, 
 such that $(R_0-\varepsilon_{0},\cdots, R_0-\varepsilon_{0},(1-t_0)\sum_{\alpha=1}^{n-1}\underline{\tilde{\mathfrak{g}}}_{\alpha\bar\alpha}-\varepsilon_{0})\in \Gamma_n$,  
\begin{equation}
\label{opppp-Gauduchon}
\begin{aligned}
  \tilde{f}(R_0-\varepsilon_{0},\cdots, R_0-\varepsilon_{0},(1-t_0)\sum_{\alpha=1}^{n-1}\underline{\tilde{\mathfrak{g}}}_{\alpha\bar\alpha}-\varepsilon_{0})\geq   \psi.
\end{aligned}
\end{equation}

Note that 
\begin{equation}
{\tilde{A}}(R)={\tilde{\underline{A}}}(R)=RI_n-\left(
\begin{matrix}
\tilde{\mathfrak{{g}}}_{1\bar 1}&&  &\tilde{\mathfrak{g}}_{1 \bar n}\\
&\ddots&&\vdots \\
& &  \tilde{\mathfrak{{g}}}_{{(n-1)} \overline{(n-1)}}& \tilde{\mathfrak{g}}_{(n-1) \bar n}\\
\tilde{\mathfrak{g}}_{n \bar 1}&\cdots& \tilde{\mathfrak{g}}_{n \overline{(n-1)}}&
R-(1-t_0) \sum_{\alpha=1}^{n-1}\tilde{\mathfrak{\underline{g}}}_{\alpha \bar\alpha}  \nonumber
\end{matrix}
\right).
\end{equation}
here $I_n= \left(\delta_{ij}\right)$.
  Let's pick  $\epsilon=\frac{\varepsilon_0(1-t_0)}{2(n-1)}$ in Lemma \ref{yuan's-quantitative-lemma} and set
\begin{equation}
\begin{aligned}
 R_s=\,& \frac{2(n-1)(2n-3)}{\varepsilon_0(1-t_0)}
\sum_{\alpha=1}^{n-1} | \tilde{\mathfrak{g}}_{\alpha \bar n}|^2
+ (n-1)\sum_{\alpha=1}^{n-1} | \tilde{\mathfrak{{g}}}_{\alpha \bar\alpha}| 
    + (1-t_0)\sum_{\alpha=1}^{n-1} | \tilde{\mathfrak{\underline{g}}}_{\alpha \bar\alpha}|
 +R_0, \nonumber
\end{aligned}
\end{equation}
where $\varepsilon_0$ and $R_0$ are fixed constants so that \eqref{opppp-Gauduchon} holds.
 Let $\lambda(\tilde{\underline{A}}(R_s))=(\lambda_1(R_s),\cdots,\lambda_n(R_s))$ be the eigenvalues of $\tilde{\underline{A}}(R_s)$.
 It follows from  Lemma  \ref{yuan's-quantitative-lemma}   
 that 
\begin{equation}
\label{lemma12-yuan-Gauduchon}
\begin{aligned}
\lambda_\alpha(R_s) \geq \,& R_s- \tilde{\mathfrak{g}}_{1\bar 1}-\frac{\varepsilon_0}{2(n-1)},
\mbox{  } \forall 1\leq \alpha<n, \\
\lambda_n(R_s) \geq \,& (1-t_0)\sum_{\alpha=1}^{n-1}\tilde{\mathfrak{\underline{g}}}_{\alpha\bar \alpha}
-\frac{\varepsilon_0}{2}.  \nonumber
\end{aligned}
\end{equation}
Therefore 
 \begin{equation}
\label{puretangential2-gauduchon}
\begin{aligned}
\tilde{f}(\lambda(\tilde{{A}}(R_s)))\geq \psi. \nonumber
\end{aligned}
\end{equation}
  We get
   $$\mathrm{tr}_\omega(\tilde{\mathfrak{g}}) \leq R_s.$$
   
     Consequently,  together with Lemma \ref{key-lemma-B1} and
    Proposition \ref{mix-general-2}, we derive Proposition \ref{proposition-quar-yuan2} and so the following quantitative boundary estimate.
   \begin{theorem}
   Under the assumptions of  Theorem \ref{thm0-n-1-yuan3}, for any $(n-1)$-PSH function $u\in C^3(M)\cap C^2(\bar M)$ solving the Dirichlet problem \eqref{MA-n-1} and \eqref{bdy-value-2}, we have
   \[\sup_{\partial M}\Delta u\leq C(1+|\nabla u|^2).\]
   \end{theorem}
   
   \subsection{Further discussion}
 The results above are valid for  more general equations 
\begin{equation}
\label{mainequ-gauduchon-general**}
\begin{aligned}
 f(\lambda(*\Phi[u]))=\psi \mbox{ in } M,     \mbox{   }
u=\varphi \mbox{ on }   \partial M 
\end{aligned}
\end{equation}
on compact Hermitian manifolds with smooth boundary, 
 where $$*\Phi[u]=\chi+\Delta u \omega-\sqrt{-1}\partial\overline{\partial}u+\varrho Z[u],$$  
 and $\varrho$ is a smooth function, i.e.
 $$\Phi[u] =*\chi+\frac{1}{(n-2)!}\sqrt{-1}\partial\overline{\partial}u\wedge\omega^{n-2}+\frac{\varrho}{(n-1)!}\mathfrak{Re}(\sqrt{-1}\partial u\wedge \overline{ \partial}\omega^{n-2}).$$
  In addition to \eqref{elliptic}, \eqref{concave}, \eqref{addistruc}, we further assume
    \begin{equation}
    \label{unbounded-2}
    \begin{aligned}
\,&  \lim_{t\rightarrow+\infty}  f(\lambda_1+t,\cdots,\lambda_{n-1}+t,\lambda_n)=\sup_\Gamma f, \,& \forall \lambda\in\Gamma.
    \end{aligned}
    \end{equation}

 The case $\Gamma\neq\Gamma_n$ is completely solved in \cite{yuan2019},  since $\Gamma_{\mathbb{R}^1}^\infty=\mathbb{R}$ for $\Gamma\neq\Gamma_n$,
 where $$\Gamma_{\mathbb{R}^1}^\infty=\{c\in\mathbb{R}: (c,R,\cdots,R)\in\Gamma \mbox{ for some } R>0\}.$$
  Thus it requires only to consider the case $\Gamma=\Gamma_n$.
 \begin{theorem}
 Let $(M,J,\omega)$ be a compact Hermitian manifold  with smooth boundary,  we assume $\varphi$, $\psi$ are all smooth and satisfies \eqref{nondegenerate}.
 Suppose $f$ satisfies \eqref{elliptic}, \eqref{concave}, \eqref{addistruc} and \eqref{unbounded-2}. Suppose there is a $C^{2,1}$ subsolution with 
 \begin{equation}
\begin{aligned}
 f(\lambda(*\Phi[\underline{u}]))\geq\psi, \mbox{ } *\Phi[\underline{u}]>0 \mbox{ in } \bar M,     \mbox{   }
\underline{u}=\varphi \mbox{ on }   \partial M. 
\end{aligned}
\end{equation}
 Then Dirichlet problem \eqref{mainequ-gauduchon-general**} admits a unique smooth solution with $*\Phi[u]>0$ in $\bar M$.
 \end{theorem}


    \begin{appendix}

\section{Key lemmas}   \label{appendix1}

    The  following two lemmas proposed in earlier works \cite{yuan2017,yuan2019} are key ingredients in proof of Propositions \ref{proposition-quar-yuan1} and \ref{proposition-quar-yuan2}. 
    
    \subsection{A characterization of concave function satisfying \eqref{addistruc}}
     \begin{lemma} 
[{\cite{yuan2019}}]
\label{asymptoticcone1}
If $f$ satisfies \eqref{elliptic} and \eqref{concave}, then the following three statements are equivalent each other.
\begin{itemize}
 \item[$\mathrm{\bf (a)}$] f satisfies \eqref{addistruc}.
 \item[$\mathrm{\bf (b)}$]  $\sum_{i=1}^n f_i(\lambda)\mu_i>0 \mbox{ for any } \lambda, \mu\in \Gamma$. 
\end{itemize}
\end{lemma}

    \begin{proof}

The concavity of $f$ means
 \begin{equation}\label{concave1}\begin{aligned}
\sum_{i=1}^n f_i(\lambda)(\mu_i-\lambda_i)\geq f(\mu)-f(\lambda) \mbox{ for $\lambda$, $\mu\in\Gamma$}
\end{aligned}\end{equation}

 $\mathrm{\bf (b)}\Rightarrow \mathrm{\bf (a)}$  
 For any $\lambda$, $\mu\in\Gamma$,  
    $t\lambda-\mu\in\Gamma$ for some $t\gg1$.
 Thus  $f(t\lambda)> f(\mu)$ for such $t$.
   
$\mathrm{\bf (a)}\Rightarrow \mathrm{\bf (b)}$ Fix $\lambda\in \Gamma$. The condition \eqref{addistruc} implies that for any  
$\mu \in \Gamma$, there is $T\geq1$ (may depend on $\mu$) such that for each $t>T$,
 $f(t\mu)>f(\lambda)$. Together with \eqref{concave1},
 one gets $\sum_{i=1}^n f_i(\lambda) (t\mu_i-\lambda_i)>0$. Thus, $\sum_{i=1}^nf_i(\lambda)\lambda_i>0$ (if one takes $\mu=\lambda$)
  then $\sum_{i=1}^n f_i(\lambda)\mu_i>0$.

 \end{proof}
 
 Lemma \ref{asymptoticcone1} implies that for any $n\times n$ Hessian matrices 
 $A=(A_{i\bar j})$, $B=(B_{i\bar j})$ with $\lambda(A)\in \Gamma$ and $\lambda(B)\in\Gamma$, 
 \begin{equation}
 \label{key-01-yuan3}
    \begin{aligned}
    \frac{\partial F}{\partial A_{i\bar j}} (A)B_{i\bar j} >0
    \end{aligned}
    \end{equation}
    where we denote $F(A)=f(\lambda(A))$.
  Consequently, \eqref{key-01-yuan3} and \eqref{concave1} imply
 \begin{equation}
 \label{concavity2}
    \begin{aligned}
    F(A+B)>F(A) 
    \end{aligned}
    \end{equation}
     for any $A$, $B$ satisfying $\lambda(A)$, $\lambda(B)\in\Gamma$.

     \subsection{A quantitative lemma}

     \begin{lemma}
     [\cite{yuan2017}]
\label{yuan's-quantitative-lemma}
Let $A$ be an $n\times n$ Hermitian matrix
\begin{equation}\label{matrix3}\left(\begin{matrix}
d_1&&  &&a_{1}\\ &d_2&& &a_2\\&&\ddots&&\vdots \\ && &  d_{n-1}& a_{n-1}\\
\bar a_1&\bar a_2&\cdots& \bar a_{n-1}& \mathrm{{\bf a}} 
\end{matrix}\right)\end{equation}
with $d_1,\cdots, d_{n-1}, a_1,\cdots, a_{n-1}$ fixed, and with $\mathrm{{\bf a}}$ variable.
Denote $\lambda=(\lambda_1,\cdots, \lambda_n)$ by   the eigenvalues of $A$.
Let $\epsilon>0$ be a fixed constant.
Suppose that  the parameter $\mathrm{{\bf a}}$ in $A$ satisfies  the quadratic
 growth condition  
  \begin{equation}
 \begin{aligned}
\label{guanjian1-yuan}
\mathrm{{\bf a}}\geq \frac{2n-3}{\epsilon}\sum_{i=1}^{n-1}|a_i|^2 +(n-1)\sum_{i=1}^{n-1} |d_i|+ \frac{(n-2)\epsilon}{2n-3},
\end{aligned}
\end{equation}
where  $\epsilon$ is a positive constant.
 Then the eigenvalues (possibly with an order) behavior like
\begin{equation}
\begin{aligned}
\,& |d_{\alpha}-\lambda_{\alpha}|
<   \epsilon, \forall 1\leq \alpha\leq n-1,\\
\,&0\leq \lambda_{n}-\mathrm{{\bf a}}
< (n-1)\epsilon. \nonumber
\end{aligned}
\end{equation}
\end{lemma}


\vspace{2mm}
For convenience we will give the proof of Lemma \ref{yuan's-quantitative-lemma}. 
We start with the case of $n=2$. In this case,
we prove that  if $\mathrm{{\bf a}} \geq \frac{|a_1|^2}{ \epsilon}+ d_1$
 then
$$0\leq d_1- \lambda_1=\lambda_2-\mathrm{{\bf a}} <\epsilon.$$

Let's briefly  present the discussion as follows:
For $n=2$, the eigenvalues of $\mathrm{A}$ are
 $\lambda_{1}=\frac{\mathrm{{\bf a}}+d_1- \sqrt{(\mathrm{{\bf a}}-d_1)^2+4|a_1|^2}}{2}$
 and $\lambda_2=\frac{\mathrm{{\bf a}}+d_1+\sqrt{(\mathrm{{\bf a}}-d_1)^2+4|a_1|^2}}{2}$.
We can assume $a_1\neq 0$; otherwise we are done.
If $\mathrm{{\bf a}} \geq \frac{|a_1|^2}{ \epsilon}+ d_1$ then one has
\begin{equation}
\begin{aligned}
0\leq d_1- \lambda_1 =\lambda_2-\mathrm{{\bf a}}
= \frac{2|a_1|^2}{\sqrt{ (\mathrm{{\bf a}}-d_1)^2+4|a_1|^2 } +(\mathrm{{\bf a}}-d_1)}
< \frac{|a_1|^2}{\mathrm{{\bf a}}-d_1 } \leq \epsilon.   \nonumber
\end{aligned}
\end{equation}
Here we use $a_1\neq 0$ to verify that the strictly inequality in the above formula holds.
We hence obtain Lemma \ref{yuan's-quantitative-lemma}  for $n=2$.

The following lemma enables us  to count  the eigenvalues near the diagonal elements
via a deformation argument.
It is an essential  ingredient in the proof of  Lemma \ref{yuan's-quantitative-lemma}   for general $n$.
\begin{lemma}
[\cite{yuan2017}]
\label{refinement}
Let $\mathrm{A}$ be an $n\times n$  Hermitian matrix
\begin{equation}
\label{matrix2}
\left(
\begin{matrix}
d_1&&  &&a_{1}\\
&d_2&& &a_2\\
&&\ddots&&\vdots \\
&& &  d_{n-1}& a_{n-1}\\
\bar a_1&\bar a_2&\cdots& \bar a_{n-1}& \mathrm{{\bf a}} \nonumber
\end{matrix}
\right)
\end{equation}
with $d_1,\cdots, d_{n-1}, a_1,\cdots, a_{n-1}$ fixed, and with $\mathrm{{\bf a}}$ variable.
Denote
$\lambda_1,\cdots, \lambda_n$ by the eigenvalues of $\mathrm{A}$ with the order
$\lambda_1\leq \lambda_2 \leq\cdots \leq \lambda_n$.
Fix   a positive constant $\epsilon$.
Suppose that the parameter $\mathrm{{\bf a}}$ in the matrix $\mathrm{A}$ 
satisfies  the following quadratic growth condition
\begin{equation}
\label{guanjian2}
\begin{aligned}
\mathrm{{\bf a}} \geq \frac{1}{\epsilon}\sum_{i=1}^{n-1} |a_i|^2+\sum_{i=1}^{n-1}  [d_i+ (n-2) |d_i|]+ (n-2)\epsilon.
\end{aligned}
\end{equation}
Then for   any $\lambda_{\alpha}$ $(1\leq \alpha\leq n-1)$ there exists an  $d_{i_{\alpha}}$
with lower index $1\leq i_{\alpha}\leq n-1$ such that
\begin{equation}
\label{meishi}
\begin{aligned}
 |\lambda_{\alpha}-d_{i_{\alpha}}|<\epsilon,
\end{aligned}
\end{equation}
\begin{equation}
\label{mei-23-shi}
0\leq \lambda_{n}-\mathrm{{\bf a}} <(n-1)\epsilon + |\sum_{\alpha=1}^{n-1}(d_{\alpha}-d_{i_{\alpha}})|.
\end{equation}
\end{lemma}

\begin{proof}
Without loss of generality, we assume $\sum_{i=1}^{n-1} |a_i|^2>0$ and  $n\geq 3$
(otherwise we are done, since $\mathrm{A}$ is diagonal or $n=2$).
Note that in the assumption of the lemma the eigenvalues have
the order $\lambda_1\leq \lambda_2\leq \cdots \leq \lambda_n$.
It is  well known that, for a Hermitian matrix,
 any diagonal element is   less than or equals to   the  largest eigenvalue.
 In particular,
 \begin{equation}
 \label{largest-eigen1}
 \lambda_n \geq \mathrm{{\bf a}}.
 \end{equation}

We only need to prove   \eqref {meishi}, since  \eqref{mei-23-shi} is a consequence of  \eqref{meishi}, \eqref{largest-eigen1}  and
\begin{equation}
\label{trace}
 \sum_{i=1}^{n}\lambda_i=\mbox{tr}(\mathrm{A})=\sum_{\alpha=1}^{n-1} d_{\alpha}+\mathrm{{\bf a}}.
 \end{equation}

 Let's denote   $I=\{1,2,\cdots, n-1\}$. We divide the index set   $I$ into two subsets  by
$${\bf B}=\{\alpha\in I: |\lambda_{\alpha}-d_{i}|\geq \epsilon, \mbox{   }\forall i\in I\} $$
and $ {\bf G}=I\setminus {\bf B}=\{\alpha\in I: \mbox{There exists an $i\in I$ such that }
 |\lambda_{\alpha}-d_{i}| <\epsilon\}.$

To complete the proof we need to prove ${\bf G}=I$ or equivalently ${\bf B}=\emptyset$.
  It is easy to see that  for any $\alpha\in {\bf G}$, one has
   \begin{equation}
   \label{yuan-lemma-proof1}
   \begin{aligned}
   |\lambda_\alpha|< \sum_{i=1}^{n-1}|d_i| + \epsilon.
   \end{aligned}
   \end{equation}

   Fix $ \alpha\in {\bf B}$,  we are going to give the estimate for $\lambda_\alpha$.
The eigenvalue $\lambda_\alpha$ satisfies
\begin{equation}
\label{characteristicpolynomial}
\begin{aligned}
(\lambda_{\alpha} -\mathrm{{\bf a}})\prod_{i=1}^{n-1} (\lambda_{\alpha}-d_i)
= \sum_{i=1}^{n-1} (|a_{i}|^2 \prod_{j\neq i} (\lambda_{\alpha}-d_{j})).
\end{aligned}
\end{equation}
By the definition of ${\bf B}$, for  $\alpha\in {\bf B}$, one then has $|\lambda_{\alpha}-d_i|\geq \epsilon$ for any $i\in I$.
We therefore derive
\begin{equation}
\begin{aligned}
|\lambda_{\alpha}-\mathrm{{\bf a}} |=  \left|\sum_{i=1}^{n-1} \frac{|a_i|^2}{\lambda_{\alpha}-d_{i}}\right|\leq\sum_{i=1}^{n-1} \frac{|a_i|^2}{|\lambda_{\alpha}-d_{i}|}\leq
\frac{1}{\epsilon}\sum_{i=1}^{n-1} |a_i|^2, \mbox{ if } \alpha\in {\bf B}.
\end{aligned}
\end{equation}
Hence,  for $\alpha\in {\bf B}$, we obtain
\begin{equation}
\label{yuan-lemma-proof2}
\begin{aligned}
 \lambda_\alpha \geq \mathrm{{\bf a}}-\frac{1}{\epsilon}\sum_{i=1}^{n-1} |a_i|^2.
\end{aligned}
\end{equation}

For a set ${\bf S}$, we denote $|{\bf S}|$ the  cardinality of ${\bf S}$.
We shall use proof by contradiction to prove  ${\bf B}=\emptyset$.
Assume ${\bf B}\neq \emptyset$.
Then $|{\bf B}|\geq 1$, and so $|{\bf G}|=n-1-|{\bf B}|\leq n-2$. 

In the case of ${\bf G}\neq \emptyset$, we compute the trace of the matrix $A$ as follows:
\begin{equation}
\begin{aligned}
\mbox{tr}(\mathrm{A})=\,&
\lambda_n+
\sum_{\alpha\in {\bf B}}\lambda_{\alpha} + \sum_{\alpha\in  {\bf G}}\lambda_{\alpha}\\
> \,&
\lambda_n+
|{\bf B}| (\mathrm{{\bf a}}-\frac{1}{\epsilon}\sum_{i=1}^{n-1} |a_i|^2 )-|{\bf G}| (\sum_{i=1}^{n-1}|d_i|+\epsilon ) \\
\geq \,&
 2\mathrm{{\bf a}}-\frac{1}{\epsilon}\sum_{i=1}^{n-1} |a_i|^2 -(n-2) (\sum_{i=1}^{n-1}|d_i|+\epsilon )
\\
\geq \,& \sum_{i=1}^{n-1}d_i +\mathrm{{\bf a}}= \mbox{tr}(\mathrm{A}),
\end{aligned}
\end{equation}
where we use  \eqref{guanjian2},   \eqref{largest-eigen1}, \eqref{yuan-lemma-proof1} and \eqref{yuan-lemma-proof2}.
This is a contradiction.

In the case of ${\bf G}=\emptyset$, one knows that
\begin{equation}
\begin{aligned}
\mbox{tr}(\mathrm{A})
\geq
\mathrm{{\bf a}}+
(n-1) (\mathrm{{\bf a}}-\frac{1}{\epsilon}\sum_{i=1}^{n-1} |a_i|^2 )
 >   \sum_{i=1}^{n-1}d_i +\mathrm{{\bf a}}= \mbox{tr}(\mathrm{A}).
\end{aligned}
\end{equation}
Again, it is a contradiction.

We now prove ${\bf B}=\emptyset$.
Therefore,   ${\bf G}=I$ and  the proof is complete.
\end{proof}

We apply Lemma \ref{refinement} to prove Lemma \ref{yuan's-quantitative-lemma} via a deformation argument.

\begin{proof}
[Proof of Lemma \ref{yuan's-quantitative-lemma}]
Without loss of generality,  we assume $n\geq 3$ and  $\sum_{i=1}^{n-1} |a_i|^2>0$
 (otherwise  $n=2$ or the matrix $\mathrm{A}$ is diagonal, and then we are done).
Fix $a_1, \cdots, a_{n-1}$,
$d_1, \cdots, d_{n-1}$. 
Denote $\lambda_1(\mathrm{{\bf a}}), \cdots, \lambda_n(\mathrm{{\bf a}})$ by
 the eigenvalues of $\mathrm{A}$ with
  the order  $\lambda_1(\mathrm{{\bf a}})\leq \cdots\leq \lambda_n(\mathrm{{\bf a}})$. 
  Clearly,  the eigenvalues $\lambda_i(\mathrm{{\bf a}})$ are all continuous functions 
  in $\mathrm{{\bf a}}$.
  For simplicity, we write $\lambda_i=\lambda_i(\mathrm{{\bf a}})$. 

Fix $\epsilon>0$.  Let $I'_\alpha=(d_\alpha-\frac{\epsilon}{2n-3}, d_\alpha+\frac{\epsilon}{2n-3})$ and
$$P_0'=\frac{2n-3}{\epsilon}\sum_{i=1}^{n-1} |a_i|^2+ (n-1)\sum_{i=1}^{n-1} |d_i|+ \frac{(n-2)\epsilon}{2n-3}.$$
  In what follows we assume  $\mathrm{{\bf a}}\geq P_0'$ (i.e. \eqref{guanjian1-yuan} holds).
The connected components of $\bigcup_{\alpha=1}^{n-1} I_{\alpha}'$ are as in the following:
$$J_{1}=\bigcup_{\alpha=1}^{j_1} I_\alpha',
J_2=\bigcup_{\alpha=j_1+1}^{j_2} I_\alpha'  \cdots, J_i =\bigcup_{\alpha=j_{i-1}+1}^{j_i} I_\alpha', \cdots, 
 J_{m} =\bigcup_{\alpha=j_{m-1}+1}^{n-1} I_\alpha'.$$
 (Here we denote $j_0=0$ and $j_m=n-1$).
 Moreover
   \begin{equation}
   \begin{aligned}
J_i\bigcap J_k=\emptyset, \mbox{ for }   1\leq i<k\leq m. \nonumber
\end{aligned}
\end{equation}

Let  $$ \mathrm{{\bf \widetilde{Card}}}_k:[P_0',+\infty)\rightarrow \mathbb{N}$$
be the function that counts the eigenvalues which lie in $J_k$.
   (Note that when the eigenvalues are not distinct,  the function $\mathrm{{\bf \widetilde{Card}}}_k$ denotes  the summation of all the multiplicities of  distinct eigenvalues which
 lie in $J_k$).
  This function measures the number of the  eigenvalues which lie in $J_k$.
  
  The crucial ingredient is that  Lemma \ref{refinement}  yields the continuity of   $\mathrm{{\bf \widetilde{Card}}}_i(\mathrm{{\bf a}})$ for $\mathrm{{\bf a}}\geq P_0'$. More explicitly,
by using  Lemma \ref{refinement} and  $$\lambda_n \geq {\bf a}\geq P_0'>\sum_{i=1}^{n-1}|d_i|+\frac{\epsilon}{2n-3}$$ we conclude that
 if   $\mathrm{{\bf a}}$ satisfies the quadratic growth condition \eqref{guanjian1-yuan} then
   \begin{equation}
  \label{yuan-lemma-proof5}
  \begin{aligned}
   \,& \lambda_n \in \mathbb{R}\setminus (\bigcup_{k=1}^{n-1} \overline{I_k'})
   =\mathbb{R}\setminus (\bigcup_{i=1}^m \overline{J_i}), \\
  \,& \lambda_\alpha \in \bigcup_{i=1}^{n-1} I_{i}'=\bigcup_{i=1}^m J_{i} \mbox{ for } 1\leq\alpha\leq n-1.
  \end{aligned}
  \end{equation}
Hence,  $\mathrm{{\bf \widetilde{Card}}}_i(\mathrm{{\bf a}})$ is a continuous function
 in the variable $\mathrm{{\bf a}}$. So it is a constant.
 Together with  the line of the proof   of \cite[Lemma 1.2]{CNS3}
   we see
 that $ \mathrm{{\bf \widetilde{Card}}}_i(\mathrm{{\bf a}}) =j_i-j_{i-1}$ for sufficiently large $\mathrm{{\bf a}}$.
The constant of $ \mathrm{{\bf \widetilde{Card}}}_i$  therefore follows that
$$ \mathrm{{\bf \widetilde{Card}}}_i(\mathrm{{\bf a}})
=j_i-j_{i-1}.$$
We thus know that the   $(j_i-j_{i-1})$ eigenvalues
$$\lambda_{j_{i-1}+1}, \lambda_{j_{i-1}+2}, \cdots, \lambda_{j_i}$$
lie in the connected component $J_{i}$.
Thus, for any $j_{i-1}+1\leq \gamma \leq j_i$,  we have $I_\gamma'\subset J_i$ and  $\lambda_\gamma$
   lies in the connected component $J_{i}$.
Therefore,
$$|\lambda_\gamma-d_\gamma| < \frac{(2(j_i-j_{i-1})-1) \epsilon}{2n-3}\leq \epsilon.$$
Here we also use the fact that $d_\gamma$ is midpoint of  $I_\gamma'$ and 
every $J_i\subset \mathbb{R}$ is an open subset.

To be brief,  if for fixed index $1\leq i\leq n-1$ the eigenvalue $\lambda_i(P_0')$ lies in $J_{\alpha}$ for some $\alpha$, 
then  Lemma \ref{refinement} implies that, for any ${\bf a}>P_0'$, the corresponding eigenvalue  $\lambda_i({\bf a})$ lies in the same  interval $J_{\alpha}$.
Adapting the line of the proof   \cite[Lemma 1.2]{CNS3} to our context,
 we get the asymptotic behavior as $\mathrm{\bf a}$ goes to infinity.
 

\end{proof}

\end{appendix}

\end{document}